\documentclass[10pt,reqno]{amsart}
\usepackage{amssymb,amsmath,amscd,amsthm,amsfonts,wasysym,mathrsfs,amsxtra}
\usepackage{mathtools}
\usepackage{marginnote}
\usepackage{appendix}
\usepackage{hyperref}
\usepackage{tikz}
\usepackage[all,cmtip]{xy}
\usepackage{tikz-cd}

\newcommand{\s}{\mathbb{S}}

\newcommand{\CC}{\mathbb{C}}

\newcommand{\NN}{\mathbb{N}}
\newcommand{\PP}{\mathbb{P}}

\newcommand{\Aa}{\mathcal{A}}
\newcommand{\Bb}{\mathcal{B}}
\newcommand{\Cc}{\mathcal{C}}
\newcommand{\Ee}{\mathcal{E}}
\newcommand{\Gg}{\mathcal{G}}
\newcommand{\cg}{\mathfrak{g}}

\newcommand{\Dd}{\mathcal{D}}
\newcommand{\Ff}{\mathcal{F}}
\newcommand{\Hh}{\mathcal{H}}
\newcommand{\Kk}{\mathcal{K}}

\newcommand{\Pp}{\mathcal{P}}
\newcommand{\Oo}{\mathcal{O}}
\newcommand{\Uu}{\mathcal{U}}

\newcommand{\Qq}{\mathcal{Q}}
\newcommand{\Xx}{\mathcal{X}}
\newcommand{\Yy}{\mathcal{Y}}
\newcommand{\R}{\mathcal{R}}
\newcommand{\Ss}{\mathcal{S}}
\newcommand{\V}{\mathcal{V}}
\newcommand{\W}{\mathcal{W}}
\newcommand{\Zz}{\mathcal{Z}}

\newcommand{\GGG}{\mathscr{G}}

\newcommand{\ZZ}{\mathbb{Z}}

\newcommand{\E}{\sf{E}}
\newcommand{\F}{\sf{F}}
\newcommand{\TTt}{\sf{T}}

\newcommand{\SOD}{\sf{SOD}}
\newcommand{\Hhh}{\sf{H}}

\newcommand{\spi}{\sf{\Psi}}
\newcommand{\SL}{\mathfrak{sl}}
\newcommand{\kk}{\underline{k}}
\newcommand{\Ll}{\underline{l}}
\newcommand{\bo}{\boldsymbol{1}}
\newcommand{\blam}{\boldsymbol{\lambda}}
\newcommand{\bmu}{\boldsymbol{\mu}}
\newcommand{\Vect}{\mathrm{Vect}}
\newcommand{\tdim}{\mathrm{dim}}

\newcommand{\Hom}{\mathrm{Hom}}

\newcommand{\Ext}{\mathrm{Ext}}
\newcommand{\Sym}{\mathrm{Sym}}
\newcommand{\tdet}{\mathrm{det}}
\newcommand{\Rm}{\mathrm}

\newcommand{\GLL}{\mathrm{GL}}

\newcommand{\Gr}{\mathrm{Gr}}
\newcommand{\Fl}{\mathrm{Fl}}
\newcommand{\rk}{\text{rank}}
\newcommand{\sub}{\mathrm{sub}}
\newcommand{\quo}{\mathrm{quo}}
\newcommand{\coker}{\mathrm{coker}}
\newcommand{\mini}{\mathrm{min}}

\newtheorem{theorem}{Theorem}[section]

\newtheorem{lemma}[theorem]{Lemma}

\newtheorem{proposition}[theorem]{Proposition}

\newtheorem{corollary}[theorem]{Corollary}
\theoremstyle{definition}
\newtheorem{definition}[theorem]{Definition}

\theoremstyle{remark}
\newtheorem{remark}[theorem]{Remark}

\numberwithin{equation}{section}

\title[Semiorthogonal decomposition via categorical action]{Semiorthogonal decomposition via categorical action}

\begin{document}
	
\emergencystretch 3em
	
\address{Academia Sinica} \email{youhunghsu@gate.sinica.edu.tw}

\author[You-Hung Hsu]{You-Hung Hsu}
	
	
\makeatletter
\@namedef{subjclassname@2020}{%
	\textup{2020} Mathematics Subject Classification}
\makeatother
	
	
\maketitle
	
\begin{abstract}
	We show that the categorical action of the shifted $q=0$ affine algebra can be used to construct semiorthogonal decomposition on the weight categories. In particular, this construction recovers Kapranov's exceptional collection when the weight categories are the derived categories of coherent sheaves on Grassmannians and $n$-step partial flag varieties. Finally, as an application, we use this result to construct a semiorthogonal decomposition on the derived categories of coherent sheaves on Grassmannians of a coherent sheaf with homological dimension $\leq 1$ over a smooth projective variety $X$. 
\end{abstract}

\section{Introduction}

\subsection{Overview of semiorthogonal decompostions}
Let $X$ be a smooth complex projective variety. The bounded derived category of coherent sheaves, denoted by $\Dd^b(X)$, is an essential invariant of $X$. One fundamental method to study the structure of $\Dd^b(X)$ is via semiorthogonal decomposition, which divides a triangulated category into simpler pieces.

One way to produce a semiorthogonal decomposition is from an exceptional collection. The simplest example is Beilinson's collection for complex projective space $\PP^n$ \cite{Be}, which is later generalized to Grassmannians and type A partial flag varieties by Kapranov \cite{Kap85}, \cite{Kap88} and  Kuznetsov-Polishchuk \cite{KP} for isotropic partial flag varieties. For other constructions and examples, we refer to Kuznetsov's ICM address \cite{Ku1}, \cite{Ku2}.



\subsection{Categorical actions}
During the past decade, categorification has been an active research topic in representation theory and related areas. One of the important questions is lifting the representations of Lie algebras/quantum groups from vector spaces to categories. Such a notion is called the \textit{categorical actions}. 

Initiating from the work of Chuang-Rouquier \cite{CR} for the $\SL_{2}$-categorification, people have extensively studied the categorical actions of $\SL_{2}$ or $\Uu_{q}(\SL_{2})$ in several flavours. In particular, Cautis-Kamnitzer-Licata \cite{CKL1} introduced the notion of geometric categorical $\SL_{2}$ action in the setting of bounded derived categories of coherent sheaves. Later they \cite{CKL2} used the geometric categorical $\SL_{2}$ action to construct the derived equivalence between derived categories of coherent sheaves on cotangent bundles to complementary Grassmannians. 

The lifting of representations to the categorical level also rises other questions. For the categorical action of semisimple or Kac-Moody Lie algebra $\cg$, usually one assigns to each weight space $V(\lambda)$ an additive category $\Cc(\lambda)$ and to
generators $e_{i}$, $f_{i}$ of $\cg$ one assigns functors $\E_{i}:\Cc(\lambda) \rightarrow \Cc(\lambda+\alpha_{i})$, $\F_{i}:\Cc(\lambda+\alpha_{i}) \rightarrow \Cc(\lambda)$ respectively. These functors are then required to satisfy certain relations analogous to those in $\cg$. For example, when $\cg=\SL_{2}$ the relation $(ef-fe)|_{V(\lambda)}=\lambda\text{Id}_{V(\lambda)}$ becomes
\begin{equation} \label{eq}
	{\E\F}|_{\Cc(\lambda)} \cong {\F\E}|_{\Cc(\lambda)} \bigoplus \text{Id}_{\Cc(\lambda)}^{\oplus \lambda} \ \text{if} \ \lambda \geq 0,
\end{equation} similarly for $\lambda \leq 0$.

The equation (\ref{eq}) is an isomorphism between functors. Understanding the natural transformations between the functors in (\ref{eq}) is a central problem in the categorification of Lie algebras/quantum groups or higher representation theory. Moreover, the natural transformations should ideally induce the isomorphism in (\ref{eq}). One answer to such a problem is given by Chuang-Rouquier \cite{CR} for the Lie algebra $\SL_{2}$ and later generalized to (simply-laced) Kac-Moody algebras $\cg$ by Khovanov-Lauda \cite{KL1}, \cite{KL2}, \cite{KL3} and Rouquier \cite{R}.

\subsection{Main results}

In \cite{Hsu}, the author defines an algebra called the shifted $q=0$ affine algebra, denoted by $\dot{\Uu}_{0,N}(L\SL_{n})$, and gives a definition of its categorical action. Then he proves that there is a categorical action of the shifted $q=0$ affine algebra on the bounded derived categories of coherent sheaves on $n$-step partial flag varieties.

In this article, we try to relate the notion of semiorthogonal decomposition to the categorical action of $\dot{\Uu}_{0,N}(L\SL_{n})$. The idea comes from the observation that the Kapranov exceptional collection can be thought of as convolutions of Fourier-Mukai kernels via the categorical action of $\dot{\Uu}_{0,N}(L\SL_{n})$.

\subsubsection{The motivating example: Grassmannians.} \label{subsub1.3}

Let $N \geq 2$ be a positive integer and fix the $N$-dimensional vector space $\CC^N$. Considering the Grassmannian of $k$-dimensional  subspaces in $\CC^N$, denoted by $\Gr(k,\CC^N)$. More precisely, $\Gr(k,\CC^N)=\{ 0 \subset V \subset \CC^N \  | \ \dim V=k  \}$. Let us recall the action of $\dot{\Uu}_{0,N}(L\SL_{2})$ on $\bigoplus_{k}\Dd^b(\Gr(k,\CC^N))$ in \cite{Hsu}. The generators  $e_{r}1_{(k,N-k)}$, $f_{s}1_{(k,N-k)}$ act on $\bigoplus_{k} \Dd^b(\Gr(k,\CC^N))$ by the following correspondence diagram 
\begin{equation}  \label{dia}
	\xymatrix{ 
		&&\Fl(k-1,k)=\{0 \overset{k-1}{\subset} V' \overset{1}{\subset} V \overset{N-k}{\subset} \CC^N \} 
		\ar[ld]_{p_1} \ar[rd]^{p_2}   \\
		& \Gr(k,\CC^N)  && \Gr(k-1,\CC^N)
	}
\end{equation}  where $\Fl(k-1,k)$ is the 3-step partial flag variety and $p_{1}$, $p_{2}$ are the natural projections. Denoting $\V, \ \V'$ to be the tautological sub-bundles on $\Fl(k-1,k)$ of rank $k$, $k-1$ respectively. Then the quotient $\V/\V'$ is a natural line bundle on $\Fl(k-1,k)$.  The generators $e_{r}1_{(k,N-k)}$ act on $\bigoplus_{k} \Dd^b(\Gr(k,\CC^N))$ by lifting to the following functors 
\begin{equation*}
{\E}_{r}\bo_{(N-k,k)}:=p_{2*}(p_{1}^{*}\otimes (\V/\V')^r):\Dd^{b}(\Gr(k,\CC^N)) \rightarrow \Dd^{b}(\Gr(k-1,\CC^N))
\end{equation*} and similarly for the lift of $f_{s}1_{(k,N-k)}$ to ${\F}_{s}\bo_{(k,N-k)}$. We refer to Section \ref{section 4} for more details.

Due to the work by Kapranov \cite{Kap85}, there are some exceptional collections on $\Dd^b(\Gr(k,\CC^N))$. One of them is given by $\{\s_{\blam}\V\}$, where $\V$ is the universal  sub-bundle of rank $k$ and $\s_{\blam}$ is the Schur functor associated to Young diagrams $\blam=(\lambda_{1},...,\lambda_{k})$ with $0 \leq \lambda_{k} \leq ... \leq \lambda_{1} \leq N-k$.  Moreover, we denote the set of such Young diagrams to be $ P(N-k,k)$.


Since $\Gr(0,\CC^N)$ is a point, $\s_{\blam}\V$ can be viewed as an object in $\Dd^b(\Gr(0,\CC^N) \times \Gr(k,\CC^N))$. Also, the functors ${\E}_{r}\bo_{(k,N-k)}$ and ${\F}_{s}\bo_{(k,N-k)}$ are Fourier-Mukai transformations with kernels denoted by $\Ee_{r}\bo_{(k,N-k)}$ and $\Ff_{s}\bo_{(k,N-k)}$, respectively. Then by the Borel-Weil-Bott theorem, we have 
\begin{equation} \label{eq 16} 
	\s_{\blam}\V \cong \Ff_{\lambda_{1}} \ast ... \ast \Ff_{\lambda_{k}} \bo_{(0,N)}  \in \Dd^b(\Gr(0,\CC^N)\times \Gr(k,\CC^N))
\end{equation} where we denote $\ast$ to be the convolution product of Fourier-Mukai kernels.

 
From (\ref{eq 16}) we know that $\Ff_{\lambda_{1}} \ast ... \ast \Ff_{\lambda_{k}} \bo_{(0,N)}$ is a 
Fourier-Mukai kernel of the functor 
\begin{equation*}
{\F}_{\blam}\bo_{(0,N)}:={\F}_{\lambda_{1}}...{\F}_{\lambda_{k}}\bo_{(0,N)} \in \Hom(\Dd^b(\Gr(0,\CC^N)),\Dd^b(\Gr(k,\CC^N))),
\end{equation*} and since the categorical action of $\dot{\Uu}_{0,N}(L\SL_{2})$ can be defined abstractly; it is natural to ask the following question  \\

$\boldsymbol{Question:}$ Given a categorical $\dot{\Uu}_{0,N}(L\SL_{2})$ action $\Kk$. Do the functors 
\begin{equation*}
\{{\F}_{\blam}\bo_{(0,N)}:={\F}_{\lambda_{1}}...{\F}_{\lambda_{k}}\bo_{(0,N)}\}_{\blam \in P(N-k,k)}
\end{equation*} behave like an exceptional collection in the triangulated category $\Hom(\Kk(0,N),\Kk(k,N-k))$?

\subsubsection{Statement of the main results}
It turns out that the answer is affirmative except that the functors $\{ {\F}_{\blam}\bo_{(0,N)}\}_{\lambda \in P(N-k,k)}$ are not exceptional. Since the above discussion of motivating example can be generalized from Grassmannians to $n$-step partial flag varieties $\Fl_{\kk}(\CC^N)$ where $\kk=(k_{1},...,k_{n}) \in \NN^n$ with $\sum_{i=1}^{n}k_{i}=N$. We state our main result in full generality.

\begin{theorem} (Theorem \ref{Proposition 4}) \label{Proposition 5}
Given a categorical $\dot{\Uu}_{0,N}(L\SL_n)$ action $\Kk$. Considering the functors ${\F}_{1,\blam(1)}{\F}_{2,\blam(2)} ... {\F}_{n-1, \blam(n-1)} \bo_{\eta} \in \Hom(\Kk(\eta),\Kk(\kk))$ where 
\begin{equation*}
	{\F}_{i,\blam(i)}:= {\F}_{i, \lambda(i)_1} {\F}_{i, \lambda(i)_2} ...  {\F}_{i, \lambda(i)_{\Bbbk_{i}}} 
\end{equation*} with $\blam(i)=(\lambda(i)_{1},...,\lambda(i)_{\Bbbk_{i}}) \in P(k_{i+1},\Bbbk_{i})$, $\Bbbk_{i}=\sum_{j=1}^{i}k_{j}$ for all $1 \leq i \leq n-1$, and $\eta=(0,0,...,0,N)$ is the highest weight. Then they satisfy the following properties 
\begin{enumerate}
	\item $\Hom({\F}_{1,\blam(1)}{\F}_{2,\blam(2)} ... {\F}_{n-1, \blam(n-1)} \bo_{\eta},{\F}_{1,\blam(1)}{\F}_{2,\blam(2)} ... {\F}_{n-1, \blam(n-1)}\bo_{\eta}) \cong \Hom(\bo_{\eta},\bo_{\eta})$ 
	\item   $\Hom({\F}_{1,\blam(1)}{\F}_{2,\blam(2)} ... {\F}_{n-1, \blam(n-1)} \bo_{\eta},{\F}_{1,\blam(1)'}{\F}_{2,\blam(2)'} ... {\F}_{n-1, \blam(n-1)'}\bo_{\eta}) \cong 0$ if $(\blam(1),...,\blam(n-1)) <_{pl}  (\blam(1)',...,\blam(n-1)')$ where $<_{pl}$ denotes the product lexicographic order; i.e., there exist $1 \leq i \leq n-1$ such that $\blam(j)=\blam(j)'$ for all $0 \leq j \leq i-1$ and $\blam(i) <_{l} \blam(i)'$ and $<_{l}$ denotes the lexicographic order.
\end{enumerate}
\end{theorem}

We state two corollaries for the above theorem. The following is the first one, which is easily seen from property (1). 

\begin{corollary}(Corollary \ref{Corollary 2}) \label{Corollary 3}
Given a categorical $\dot{\Uu}_{0,N}(L\SL_n)$ action $\Kk$. The functors 
\begin{equation*}
{\F}_{1,\blam(1)}{\F}_{2,\blam(2)} ... {\F}_{n-1, \blam(n-1)} \bo_{\eta} \in \Hom(\Kk(\eta),\Kk(\kk))
\end{equation*} where $\blam(i)=(\lambda(i)_{1},...,\lambda(i)_{\Bbbk_{i}})  \in P(k_{i+1},\Bbbk_{i})$ for all $1 \leq i \leq n-1$, are all fully faithful.
\end{corollary}

Next, property (2) from Theorem \ref{Proposition 5} tells us that the functors ${\F}_{1,\blam(1)}{\F}_{2,\blam(2)} ... {\F}_{n-1, \blam(n-1)} \bo_{\eta}$ satisfy the semiorthogonal property. Combining with Corollary \ref{Corollary 3}, we obtain the second corollary.

\begin{corollary} (Corollary  \ref{Theorem 5})
Given a categorical $\dot{\Uu}_{0,N}(L\SL_n)$ action $\Kk$. We denote
\begin{equation*}
\textnormal{Im}{\F}_{1,\blam(1)}{\F}_{2,\blam(2)} ... {\F}_{n-1, \blam(n-1)} \bo_{\eta}
\end{equation*}to be the minimal full triangulated subcategories of $\Kk(\kk)$ generated by the class of objects which are the essential images of ${\F}_{1,\blam(1)}{\F}_{2,\blam(2)} ... {\F}_{n-1, \blam(n-1)} \bo_{\eta}$ where $\blam(i)=(\lambda(i)_{1},...,\lambda(i)_{\Bbbk_{i}})  \in P(k_{i+1},\Bbbk_{i})$
for all $1 \leq i \leq n-1$. Then we have the following semiorthogonal decomposition 
\begin{equation*}
	\Kk(\kk)=\langle \Aa(\kk), \textnormal{Im}{\F}_{1,\blam(1)}{\F}_{2,\blam(2)} ... {\F}_{n-1, \blam(n-1)} \bo_{\eta} \rangle_{\blam(i) \in P(k_{i+1},\Bbbk_{i})}
\end{equation*} where $\Aa(\kk)=\langle \textnormal{Im}{\F}_{1,\blam(1)}{\F}_{2,\blam(2)} ... {\F}_{n-1, \blam(n-1)} \bo_{\eta}\rangle_{\blam(i) \in P(k_{i+1},\Bbbk_{i})}^{\perp}$ is the orthogonal complement.
\end{corollary}

\begin{remark}
When $n=2$ and the weight categories are $\Kk(k,N-k)=\Dd^b(\Gr(k,\CC^N))$, we have $\Hom(\bo_{(0,N)},\bo_{(0,N)}) \cong \CC$. Then $\{ {\F}_{\blam}\bo_{(0,N)}\}$ recovers the Kapranov exceptional collection which mentioned in the Subsection \ref{subsub1.3}.
\end{remark}


\subsection{Application to Grassmannian of coherent sheaves}
In the final part of this paper, we give an application of the main theorem to construct semiorthogonal decomposition on the derived category of coherent sheaves on Grassmannians (more precisely, relative Quot scheme) of coherent sheaf with homological dimension $\leq 1$.

Let $X$ be a connected smooth projective variety and $\GGG$ be a coherent sheaf on $X$ of homological dimension $\leq 1$, i.e. there is a resolution $\Ee^{-1} \rightarrow \Ee^{0} \twoheadrightarrow \GGG \rightarrow 0$ where $\Ee^{-1}$, $\Ee^{0}$ are locally free.  Let $N=\rk \GGG=\rk \Ee^0 -\rk \Ee^{-1} \geq 2$ be the rank of $\GGG$. We consider the Grassmannians $\Gr(\GGG, k)$ of rank $k$ locally free quotients of $\GGG$ and its derived category $\Dd^b(\Gr(\GGG,k))$. 

This geometric space $\Gr(\GGG,k)$ appears in several literatures, and many interesting moduli spaces (related to enumerative geometry) arise via this construction. For example, Addington-Takahashi \cite{AT} use such spaces as correspondences to construct categorical $\SL_{2}$ action on derived categories of coherent sheaves on moduli space of sheaves on K3 surface (see Negut \cite{N} for related work).

Like the usual Grassmannians, constructing semiorthogonal decomposition for $\Dd^b(\Gr(\GGG,k))$ becomes a natural question. It was already done by Jiang-Leung \cite{JL} for $k=1$ (which is called the \textit{projectivization formula}) and Toda \cite{T} for general $k$ (which is called the \textit{Quot formula}).

Here we provide a different construction by using the categorical action of $\dot{\Uu}_{0,N}(L\SL_{2})$. We define functors ${\E}_{r}$, ${\F}_{s}$ as Fourier-Mukai transforms via correspondence that is similar to the diagram (\ref{dia}). Then we verify those functors satisfy the conditions in Definition \ref{definition 2} and thus we can apply Theorem \ref{Proposition 5} to obtain the following result.

\begin{theorem} [Theorem \ref{Theorem 7}]
The derived category $\Dd^b(\Gr(\GGG,k))$ admits the following two semiorthogonal decompositions
	\begin{align*}
		\Dd^b(\Gr(\GGG,k)) &= \langle \Aa(N-k,k), \ \textnormal{Im}{\F}_{\blam}\bo_{(0,N)}  \rangle_{\blam \in P(k,N-k)} \\
		&=\langle \Bb(N-k,k), \ \textnormal{Im}{\E}_{-\bmu}\bo_{(N,0)} \rangle_{\bmu \in P(N-k,k)}
	\end{align*} where $\Aa(N-k,k):=\langle \textnormal{Im}{\F}_{\blam}\bo_{(0,N)} \rangle_{\blam \in P(k,N-k)}^{\perp}$, $\Bb(N-k,k):=\langle \textnormal{Im}{\E}_{-\bmu}\bo_{(N,0)} \rangle_{\bmu \in P(N-k,k)}^{\perp}$ are the orthogonal complement.
\end{theorem}

Unlike the results by Jiang-Leung \cite{JL} and Toda \cite{T}, the orthogonal complements $\Aa(N-k,k)$, $\Bb(N-k,k)$ in our Theorem are not clear to see. In future work, we would like to understand the representation-theoretic meaning of those orthogonal complements and see the relation between our construction and Jiang-Leung, Toda.

\subsection{Some further remarks}
Finally, we give some remarks about the results and to other related works.

\subsubsection{The dual exceptional collection}
We also have the notion of dual exceptional collection (see Definition \ref{definition 9} for the definition), which is similar to the dual basis. From Theorem \ref{Theorem 1} the right dual exceptional collection for $\{\s_{\blam}\V\}$ is given by $\{\s_{\bmu}\CC^N/\V[-|\bmu|]\}$ where $\bmu=(\mu_{1},...,\mu_{N-k})$ with $0 \leq \mu_{N-k} \leq ... \leq \mu_{1} \leq k$.

Similarly, we can show that the functors
\begin{equation*} {\E}_{-\bmu}\bo_{(N,0)}:={\E}_{-\mu_{1}}...{\E}_{-\mu_{N-k}}\bo_{(N,0)} \in \Hom(\Kk(N,0),\Kk(k,N-k))
\end{equation*}  also give the same results like Proposition \ref{Proposition 5} and Corollary \ref{Corollary 3}. Thus $\{ {\E}_{-\bmu}\bo_{(N,0)} \}$ can also be used to construct a semiorthogonal decomposition of the weight category $\Kk(k,N-k)$ (see Theorem \ref{Theorem 6}), and when the weight categories are $\Dd^b(\Gr(k,\CC^N))$ we recover the (dual of the) above dual exceptional collection
\begin{equation*}
	\s_{\bmu}(\CC^N/\V)^{\vee}
	\cong \Ee_{-\mu_{1}} \ast ... \ast \Ee_{-\mu_{N-k}} \bo_{(N,0)} \in \Dd^b(\Gr(N,N) \times \Gr(k,N)).
\end{equation*} 

\subsubsection{Characteristic free}
The classical Kapranov exceptional collection had been generalized in another direction to the characteristic-free setting by Buchweitz-Leuschke-Van den Bergh \cite{BLVdB} and Efimov \cite{E}. It would be interesting to see whether our construction can also be applied in a characteristic-free setting.

\subsubsection{Comapre to other categorical actions}

Since the notion of categorical actions of Lie algebras/quantum groups (e.g., $\SL_{2}$ or $\Uu_{q}(\SL_{2})$) has been defined before, people may wonder whether those categorical actions can be used to produce semiorthogonal decompositions like the one in this article. 

To our best knowledge, it seems like there is only little to know about such a question. Also, some definitions of categorical action only require the weight categories to be abelian, for example,  \cite{CR}, where the notion of semiorthogonal decomposition does not exist.

We can address a bit about this question in the case of geometric categorical $\SL_{2}$/$\SL_{n}$ (or $\cg$) actions developed by Cautis-Kamnitzer-Liciata in the series of articles \cite{CKL1}, \cite{CKL2}, \cite{CK}. Note that the main examples of weight categories for such action are the derived categories of coherent sheaves on the cotangent bundle of Grassmannians/partial flag varieties (or Nakajima quiver varieties), and the canonical bundle for those varieties are trivial. This implies that the varieties are Calabi-Yau. Thus their derived categories of coherent sheaves do not have non-trivial semi-orthogonal decompositions.


\subsection{Notations}
For a $n$-tuple of positive integer $\kk=(k_1,...,k_n)$, we denote $\Bbbk_{i}=\sum_{j=1}^{i}k_{j}$ for all $ 1 \leq j \leq n$. A Young diagram will always be denoted by a bold symbol $\blam=(\lambda_{1},...,\lambda_{k})$. We denote $P(a,b)$ to be the set of Young diagrams $\blam$ such that $\lambda_{1} \leq a$ and $\lambda_{b+1} =0$. 

In this paper, we work over the field $\CC$ of complex numbers. For a vector space $V$, we denote $V^*=\Hom_{\CC}(V,\CC)$ to be the dual space. For a coherent sheaf $\Gg$ on an algebraic variety $X$, we denote $\Gg^{\vee}=\Hh om_{\Oo_{X}}(\Gg,\Oo_{X})$ to be the dual coherent sheaf. Let $\dim_{\CC} V=N$ and $k \leq N$. Then we denote $\Gr(k,V)$ to be the Grassmannian of $k$-dimensional subspace of $V$ and $\Gr(V,k)$ to be the Grassmannian of $k$-dimensional quotient spaces of $V$. They are related by $\Gr(k,V)=\Gr(V,N-k)$, and $\Gr(k,V)=\Gr(N-k,V^*)=\Gr(V^*,k)$. In particular, when $k=1$, we have the two projective space $\Gr(1,V)$, $\Gr(V,1)$ parametrizing $1$-dimensional subspace, quotient space, respectively. We use the notations $\PP_{\sub}(V)=\Gr(1,V)$ and $\PP_{\quo}(V)=\Gr(V,1)$ to distinguish them. Moreover, the above construction can be generalized to vector bundle $\V$ of rank $N$ on $X$, and similarly, we have the Grassmannian bundles $\Gr(k,\V)$, $\Gr(\V,k)$,...etc on $X$.

Finally, for a smooth projective variety $X$, we denote $\Dd^b(X)$ to be its bounded derived category of coherent sheaves. All functors between derived categories will be assumed to be derived. For example, we will write $f_*$ and $\otimes$ instead of $Rf_*$ and $\otimes^{L}$, respectively. 

\subsection{Acknowledgements}
The author would like to thank Wu-yen Chuang and Hsueh-Yung Lin for helpful discussions and also want to thank Chun-Ju Lai for his careful reading of the first draft and for providing useful feedback on the article. Also thanks Yu Zhao and Qingyuan Jiang for some helpful email correspondences. Finally, the author wants to thank the anonymous referee's  feedback and comments on the article.

\section{Preliminaries}
We briefly recall the definitions of semiorthogonal decompositions, Fourier-Mukai transformations, Beilinson-Kapranov exceptional collections, and the projective bundle formula that will be used in the later sections. We follow the book \cite{Huy} for the definitions, and we refer to \cite{Ku1} for a more general review of semiorthogonal decompositions.

\subsection{Semiorthogonal decompositions} 

Let $\Dd$ be a $\CC$-linear triangulated category. 

For a full triangulated subcategory $\Aa \subset \Dd$, we define \begin{align*}
	&\Aa^{\perp}=\{X \in \text{Ob}(\Dd) \ | \  \Hom_{\Dd}(A,X)=0 \  \forall  \ A \in \text{Ob}(\Aa) \}, \\
	&^{\perp}\Aa=\{X \in \text{Ob}(\Dd) \ | \  \Hom_{\Dd}(X,A)=0 \  \forall  \ A \in \text{Ob}(\Aa) \},
\end{align*} to be the right and left orthogonals to $\Aa$ in $\Dd$, respectively. Those are triangulated subcategories of $\Dd$.

Given a class $\Ee$ of objects in $\Dd$, we denoted $\langle \Ee \rangle$ to be the minimal full triangulated subcategory of $\Dd$ containing all objects in $\Ee$ and closed under taking direct summands. Similarly for any sequence of full triangulated subcategories $ \Aa_{1},...,\Aa_{n}$ in $\Dd$ we denote by $\langle \Aa_{1},...,\Aa_{n} \rangle$  the minimal full triangulated subcategory of $\Dd$ which contains all of $\Aa_{i},...,\Aa_{n}$.

\begin{definition} \label{definition 4}
	A  \textit{semiorthogonal decomposition} ($\SOD$ for short) of $\Dd$ is a sequence of full triangulated subcategories $\Aa_{1},...,\Aa_{n}$ such that \begin{enumerate}
		\item there is no non-zero Homs from right to left, i.e.
		$\Hom_{\Dd}(A_{i},A_{j})=0$ for all $A_{i} \in \text{Ob}(\Aa_{i})$, $A_{j} \in \text{Ob}(\Aa_{j})$ where $1 \leq j <  i \leq n$.
		\item $\Dd$ is generated by $\Aa_{1},...,\Aa_{n}$, i.e. the smallest full triangulated subcategory containing $\Aa_{1},...,\Aa_{n}$ is equal to $\Dd$.
	\end{enumerate}

	We will use the notation $\Dd=\langle \Aa_{1},...,\Aa_{n} \rangle$ for a semiorthogonal decomposition of $\Dd$ with components $\Aa_{1},...,\Aa_{n}$.
\end{definition}

\begin{remark}
Note that in some literatures, e.g. \cite{BOR}, the authors may require the subcategories $\Aa_{i} \subset \Dd$ in Definition \ref{definition 4} to be strictly full. However, in this article, we do not need this requirement.
\end{remark}

\begin{definition}  \label{definition 5}
A full triangulated subcategory $\Aa \subset \Dd$ is called  \textit{right admissible} if, for the inclusion functor $i:\Aa \rightarrow \Dd$, there is a right adjoint $i^{!}:\Dd \rightarrow \Aa$, and  \textit{left admissible} if there is a left adjoint $i^*:\Dd \rightarrow \Aa$. It is called  \textit{admissible} if it is both left and right admissible.
\end{definition}

\begin{lemma} [\cite{Bo}] \label{lemma 1} 
	If $\Dd=\langle \Aa, \Bb \rangle$ is a semiorthogonal decomposition, then
	$\Aa$ is left admissible and B is right admissible. Conversely, if $\Aa \subset \Dd$ is left admissible, then $\Dd=\langle \Aa, ^{\perp}\Aa \rangle$ is a semiorthogonal decomposition, and if $\Bb \subset \Dd$ is right admissible,
	then $\Dd=\langle \Bb^{\perp}, \Bb \rangle$ is a semiorthogonal decomposition.
\end{lemma}

This lemma has the following generalization.

\begin{lemma} [\cite{Bo}] \label{lemma 2} 
Let $\Aa_{1},\Aa_{2},...,\Aa_{n}$ be a semiorthogonal sequence of admissible (full triangulated) subcategories in $\Dd$; i.e., $\Hom(\Aa_{i},\Aa_{j})=0$ for all $1 \leq j < i \leq n$. Then for each $0 \leq i \leq n$ there is a $\SOD$ 
\begin{equation*}
\Dd=\langle \Aa_{1},...,\Aa_{i}, ^{\perp}\langle \Aa_{1},...,\Aa_{i} \rangle \cap \langle \Aa_{i+1},...,\Aa_{n} \rangle^{\perp}, \Aa_{i+1},...,\Aa_{n} \rangle.
\end{equation*}
\end{lemma}



The simplest example of an admissible subcategory is the one generated by an exceptional object.
\begin{definition}  \label{definition 7}
	An object $E \in \text{Ob}(\Dd)$ is called  \textit{exceptional} if 
	\[
	\Hom_{\Dd}(E,E[l])=\begin{cases}
		\CC & \ \text{if} \ l=0 \\
		0 & \ \text{if} \ l \neq 0.
	\end{cases}
	\]
\end{definition}

Then we define the notion of exceptional collections.

\begin{definition} \label{definition 8}
	An ordered collection $\{E_{1},...,E_{n}\}$, where $E_{i} \in \text{Ob}(\Dd)$ for all $1 \leq i \leq n$, is called an \textit{exceptional collection} if each $E_{i}$ is exceptional and $\Hom_{\Dd}(E_{i},E_{j}[l])=0$ for all $i>j$ and $\l \in \ZZ$.	
\end{definition}

An exceptional collection $\{E_{1},...,E_{n}\}$ in $\Dd$ naturally give rise to a $\SOD$ of $\Dd$
\begin{equation*}
	\Dd=\langle \Aa, E_{1},...,E_{n} \rangle
\end{equation*} where $\Aa=\langle E_{1},...,E_{n} \rangle^{\perp}$ and $E_{i}$ denote the full triangulated subcategory generated by the object $E_{i}$. An exceptional collection is called  \textit{full} if the subcategory $\Aa$ is zero. 

Finally, we define the notion of a (right) dual exceptional collection.
\begin{definition}  \label{definition 9}
	Let $\{E_{1},...,E_{n}\}$ be an exceptional collection on the triangulated category $\Dd$. An exceptional collection $\{F_{1},...,F_{n}\}$ of objects in $\Dd$ is called \textit{right dual} to $\{E_{1},...,E_{n}\}$ if 
	\begin{equation*}
		\Hom_{\Dd}(F_{i},E_{i}[l])=\begin{cases}
			\CC & \ \text{if} \ l=0 \\
			0  & \ \text{if} \ l \neq 0
		\end{cases}
	\end{equation*} and $\Hom_{\Dd}(F_{i},E_{j}[l])=0$ for $i \neq j$ and $l \in \ZZ$.
\end{definition} 

We will see more examples of semiorthogonal decompositions given by exceptional collections in Subsection \ref{2.3}.

\subsection{Fourier-Mukai transformations}

Next, we recall the tool of Fourier-Mukai transformations which will be used in the proof of Theorem \ref{Theorem 7} for constructing categorical action. 

\begin{definition} 
	Let $X$ and $Y$ be two smooth projective varieties over $\CC$. A  \textit{Fourier-Mukai kernel} is an object $\Pp$ of the
	bounded derived category of coherent sheaves on $X \times Y$. Given $\Pp \in \Dd^b(X \times Y)$, we may define
	the associated  \textit{Fourier-Mukai transform}, which is the functor
	\[
	\Phi_{\Pp}:\Dd^b(X) \rightarrow \Dd^b(Y) 
	\]
	\[
	\ \ \ \ \ \ \ \ \Ff \mapsto \pi_{2*}(\pi_{1}^*(\Ff) \otimes \Pp)
	\] where $\pi_{1}$, $\pi_{2}$ are natural projections from $X \times Y$ to $X$, $Y$ respectively.
\end{definition}

We call $\Phi_{\Pp}$ the Fourier-Mukai transform with (Fourier-Mukai) kernel $\Pp$. For convenience, we will just write FM for Fourier-Mukai. The first property of FM transforms is that they have left and right adjoints that are themselves FM transforms.

\begin{proposition} [\cite{Huy} Proposition 5.9] \label{Proposition 1}
	For $\Phi_{\Pp}:\Dd^b(X) \rightarrow \Dd^b(Y)$ is the FM transform with kernel $\Pp$, define 
	\[
	\Pp_{L}=\Pp^{\vee} \otimes \pi^*_{2}\omega_{Y}[\tdim  Y], \ \Pp_{R}=\Pp^{\vee} \otimes \pi^*_{1}\omega_{X}[\tdim X].
	\] 
	
	Then
	\[
	\Phi_{\Pp_{L}}:\Dd^b(Y) \rightarrow \Dd^b(X), \  \Phi_{\Pp_{R}}:\Dd^b(Y) \rightarrow \Dd^b(X)
	\] are the left and right adjoints of $\Phi_{\Pp}$, respectively.
\end{proposition}

The second property is that the composition of FM transforms is also an FM transform.

\begin{proposition} [\cite{Huy} Proposition 5.10] \label{Proposition 2}
	Let $X, Y, Z$ be smooth projective varieties over $\CC$. Consider objects $\Pp \in \Dd^b(X \times Y)$ and $\Qq \in \Dd^b(Y \times Z)$. They define FM transforms $\Phi_{\Pp}:\Dd^b(X) \rightarrow \Dd^b(Y)$, $\Phi_{\Qq}:\Dd^b(Y) \rightarrow \Dd^b(Z)$.  We use $\ast$ to denote the operation for convolution, i.e.
	\[
	\Qq \ast \Pp:=\pi_{13*}(\pi_{12}^*(\Pp)\otimes \pi_{23}^{*}(\Qq))
	\]
	
	Then for $\R=\Qq \ast \Pp \in \Dd^b(X \times Z)$, we have $\Phi_{\Qq} \circ \Phi_{\Pp} \cong \Phi_{\R}$. 
\end{proposition}

\begin{remark} \label{remark 1}
	Moreover by \cite{Huy} remark 5.11, we have $(\Qq \ast \Pp)_{L} \cong (\Pp)_{L} \ast (\Qq)_{L}$ and $(\Qq \ast \Pp)_{R} \cong (\Pp)_{R} \ast (\Qq)_{R}$.
\end{remark}

\subsection{Beilinson-Kapranov exceptional collections}  \label{2.3}

In this section, we recall the exceptional collections constructed by Beilinson for projective spaces and by Kapranov for partial flag varieties.

The first one is the projective space $\PP^N=\PP_{sub}(\CC^{N+1})=\Gr(1,\CC^{N+1})$ which is due to Beilinson.

\begin{theorem}[\cite{Be}] 
	There is a full exceptional collection
	\begin{equation*}
		\Dd^b(\PP^N)=\langle \Oo_{\PP^N}(-N), \Oo_{\PP^N}(-N+1),..., \Oo_{\PP^N} \rangle.
	\end{equation*}
\end{theorem}

Let $\blam=(\lambda_{1},...,\lambda_{n})$ be a non-increasing sequence of positive integers. We can represent $\blam$ as a Young diagram with $n$ rows, aligned on the left, such that the $i$th row
has exactly $\lambda_{i}$ cells. The size of $\blam$, denoted by $|\blam|$, is the number $|\blam|=\sum_{i=1}^{n} \lambda_{i}$. The transpose diagram $\blam^*$ is obtained by exchanging rows and columns of $\blam$. 

Next, for such  Young diagrams $\blam$ we define the notion of Schur functors.  For more details about Schur functors, we refer the readers to Chapter 4 and chapter 6 in \cite{FH}.

\begin{definition}
	Let $n \geq 1$ be a positive integer and $\blam=(\lambda_{1},...,\lambda_{n})$ be a sequence of non-increasing positive integers. The  \textit{Schur functor} $\s_{\blam}$ associated to $\blam$ is defined as a functor 
	\[
	\s_{\blam}:\Vect_{\CC} \rightarrow \Vect_{\CC}
	\] 
	such that for any vector space $V$, $\s_{\blam}V$ coincides with the image of the Young symmetrizer $c_{\blam}$ in the space of tensors of $V$ of rank $n$: i.e., $\s_{\blam}V=\text{Im} (c_{\blam}|_{V^{\otimes n}})$.
\end{definition}

Let $\Gr(k,\CC^N)=\{0 \subset V \subset \CC^N \ | \ \tdim V=k\}$ be the Grassmannian of $k$-dimensional subspaces in $\CC^N$. Denote $\V$ to be the tautological rank $k$ bundle on $\Gr(k,\CC^N)$ and $\CC^N/\V$ to be the tautological rank $N-k$ quotient bundle.  For non-negative integers $a, \ b \geq 0$, we denote by $P(a,b)$ the set of Young diagrams $\blam$ such that $\lambda_{1} \leq a$ and $\lambda_{b+1} =0$. Then we have the following theorem.

\begin{theorem} [\cite{Kap85}] \label{Theorem 1}
There is a full exceptional collection
\begin{equation*}
	\Dd^b(\mathrm{Gr}(k,\CC^N))=\langle \ \s_{\blam}\V \ \rangle_{\blam \in P(N-k,k)}.
\end{equation*}  

Its dual exceptional collection, which is also full, is given by 
\begin{equation*}
	\Dd^b(\Gr(k,\CC^N))=\langle \ \s_{\bmu}\CC^N/\V[-|\bmu|] \ \rangle_{\bmu \in P(k,N-k)}.
\end{equation*}
\end{theorem}

Here we have to mention the order in the set $P(N-k,k)$ used for the semiorthogonal property in the definition of exceptional collections. We denote $<_{\l}$ for the lexicographical order. Let $\blam=(\lambda_{1},...,\lambda_{k}),\ \blam'=(\lambda'_{1},...,\lambda'_{k}) \in P(N-k,k)$ be two different elements. We say that $\blam <_{l} \blam'$ if counting from the beginning, the order depends on the first $1 \leq j \leq k$ such that $\lambda_{j}<\lambda'_{j}$. More precisely, $\lambda_{i} \leq \lambda'_{i}$ for all $1 \leq i \leq k$ and there exist $j$ such that $\lambda_{j}<\lambda'_{j}$. For example, considering the set $P(2,2)$, we have the lexicographical order 
\begin{equation*}
	(0,0) <_{l} (1,0) <_{l} (1,1) <_{l} (2,0) <_{l} (2,1) <_{l} (2,2).
\end{equation*}

Since $\V$ is the tautological bundle on $\Gr(k,\CC^N)$, the semiorthogonal property for $\{\s_{\blam}\V\}$ should have the following statement; i.e., 
\begin{equation*}
\Hom(\s_{\blam}\V, \s_{\blam'}\V)=0
\end{equation*} if $\blam <_{l} \blam'$. For the right dual exceptional collection $\{\s_{\bmu}\CC^N/\V[-|\bmu|]\}$ the order is opposite to the order on $\{\s_{\blam}\V\}$. More precisely, we have 
\begin{equation*}
\Hom(\s_{\bmu}\CC^N/\V, \s_{\bmu'}\CC^N/\V)=0
\end{equation*} if $\bmu >_{\l} \bmu'$ where $\bmu, \bmu' \in P(k,N-k)$.

Moving to the partial flag varieties, for $\kk=(k_{1},...,k_{n}) \in \NN^n$ with $\sum_{i=1}^{n}k_{i}=N$, the partial flag variety is defined as follows 
\begin{equation*}
	\Fl_{\kk}(\CC^N):=\{V_{\bullet}=(0=V_{0} \subset V_1 \subset ... \subset V_{n}=\CC^N) \ | \ \tdim V_{i}/V_{i-1}=k_{i} \ \text{for} \ \text{all} \ i\}.
\end{equation*} 

Then we have the following generalization.
\begin{theorem}  [\cite{Kap88}]  \label{Theorem 2}
	There is a full exceptional collection
	\begin{equation*}
		\Dd^b(\Fl_{\kk}(\CC^N))=\langle \ \bigotimes_{i=1}^{n-1} \s_{\blam(i)}\V_{i} \ \rangle_{\blam(i) \in P(k_{i+1},\Bbbk_{i}), \ 1 \leq i \leq n-1}.
	\end{equation*}
\end{theorem}

\begin{remark}
The order in the partial flag variety case is the product lexicographical order on \begin{equation*}
P(k_{2},k_{1}) \times ... \times P(k_{i+1},k_{1}+..+k_{i}) \times ... \times P(k_{n},k_{1}+..+k_{n-1}),
\end{equation*} which is denoted by $<_{pl}$, where "pl" stands for product  lexicographical. More precisely, we have $(\blam(1),...,\blam(n-1)) <_{pl} (\blam(1)',...,\blam(n-1)')$ if there exists $1 \leq i \leq n-1$ such that $\blam(j)=\blam(j)'$ for all $0 \leq j \leq i-1$ and $\blam(i) <_{l} \blam(i)'$, where $\blam(i), \blam(i)' \in P(k_{i+1},\Bbbk_{i})$. 
\end{remark}

\begin{remark}
	Note that for the partial flag varieties, we do not have the right dual exceptional collection in Theorem \ref{Theorem 2}. It was pointed out in \cite{Kap88} that there are objects that satisfy the Definition \ref{definition 9} for dual exceptional collection; however, they do have higher $\Ext$'s between themself.
\end{remark}

\subsection{Projective bundle formula}
The final tool we will need in later sections is the projective bundle formula. Let $\V$ be a vector bundle of rank $n$ on a (smooth) algebraic variety $X$, where $n \geq 2$. Then we can form the projectivization $\PP_{\quo}(\V)$ which parametrizes 1-dimensional quotient bundle of $\V$. Note that from our notation we have $\PP_{\quo}(\V)= \Gr(\V,1)=\Gr(\V^{\vee},1)=\PP_{\sub}(\V^{\vee})$. We denote $\pi:\PP_{\quo}(\V) \rightarrow X$ to be the projection which is a $\PP^{n-1}$-fibration.

Let  $\Oo_{\PP_{\quo}(\V)}(-1)$ be the tautological bundle and $\Oo_{\PP_{\quo}(\V)}(1)$ be the dual bundle, and we define  $\Oo_{\PP_{\quo}(\V)}(i):=\Oo_{\PP_{\quo}(\V)}(1)^{\otimes i}$ for $i \in \ZZ$. Then we have the following result.

\begin{proposition} [Exercise 8.4 in Chapter 3 from \cite{Ha}] \label{proposition 3}
	\begin{equation*}
		\pi_{*} \Oo_{\PP_{\quo}(\V)}(i) \cong   \begin{cases} \Sym^i(\V) & \text{if} \ i \geq 0 \\ 0 & \text{if} \  1-n \leq i\leq -1 \\ \Sym^{-i-n}(\V^{\vee}) \otimes \tdet(\V)^{-1}[1-n]   & \text{if} \ i \leq -n \end{cases}. 
	\end{equation*}
\end{proposition}

The above result is generalized by Yu Zhao \cite{Z} to a two-term complex of locally free sheaves. More precisely, let $\Uu$ be a coherent sheaf on $X$ which admits a two-term locally free resolution, i.e. $\W \rightarrow \V \twoheadrightarrow \Uu \rightarrow 0$ where $\W$, $\V$ are locally free sheaves on $X$. We denote the rank of $\Uu$ to be $u=\rk \Uu= \rk \V - \rk \W \geq 1$.

Similarly we have the projectivization $\PP_{\quo}(\Uu)$ which is a closed subscheme of $\PP_{\quo}(\V)$, and  we denote $\pi_{\Uu}:\PP_{\quo}(\Uu) \rightarrow X$ to be the projection. We also have the tautological bundle $\Oo_{\PP_{\quo}(\Uu)}(-1)$ and its tensor power $\Oo_{\PP_{\quo}(\Uu)}(i)$ for $i \in \ZZ$. Then the above projective bundle formula can be generalized to the following.

\begin{proposition} [Lemma 3.9 in \cite{Z}] \label{Proposition}
\begin{equation*}
	\pi_{\Uu*} \Oo_{\PP_{\quo}(\Uu)}(i) \cong   \begin{cases} \Sym^i(\Uu) & \text{if} \ i \geq 0 \\ 0 & \text{if} \  1-u \leq i\leq -1 \\ \wedge^{-i-u}(\Uu^{\vee}) \otimes \tdet(\Uu)^{-1}[1-u]   & \text{if} \ i \leq -u \end{cases}.
\end{equation*}
\end{proposition}

\begin{remark}
More recently, the above proposition has been generalized by Qingyuan Jiang \cite{Jia2} to the derived algebraic geometry setting, i.e. $X$ is a derived scheme or prestack, which is called the generalized Serre theorem. Since we will not need such a general version in our proof, we refer the reader to Theorem 5.6 in \textit{loc. cit.}.
\end{remark}

\section{Shifted $q=0$ affine algebras and its categorical action} \label{section 4}
In this section, we recall the definitions of the shifted $q=0$ affine algebra $\dot{\Uu}_{0.N}(L\SL_{n})$ and its categorical action. We also mention the result where there is a categorical action on the bounded derived categories of coherent sheaves on partial flag varieties.

Before we go to the detailed setup and the definition of the shifted $q=0$ affine algebra, we should give the readers some background for it. The main motivation comes from the study of categorical action and categorification of the quantum group $\Uu_{q}(\cg)$ and its idempotent version $\dot{\Uu}_{q}(\cg)$, where $\cg$ a semisimple or Kac-Moody Lie algebra.

We restrict to the case $\cg=\SL_{2}$ and keep in mind that all the results can be naturally generalized to the $\SL_{n}$ case. Based on the work \cite{BLM}, Beilinson-Lusztig-MacPherson give a geometric model for $\Uu_{q}(\SL_{2})$. This can be used to construct categorical $\Uu_{q}(\SL_{2})$-action. 

The weight space $V_{\lambda}$ is replaced by the weight category $\Cc(\lambda)=\Dd^bCon(\Gr(k,N))$, which is the bounded derived category of constructible sheaves on $\Gr(k,N)$ with $\lambda=N-2k$. The generators $e$, $f$ act on them by using the following correspondence diagram

\begin{equation} \label{diag 1}
	\xymatrix{ 
		&&\Fl(k-1,k)=\{0 \overset{k-1}{\subset} V' \overset{1}{\subset} V \overset{N-k}{\subset} \CC^N \} 
		\ar[ld]_{p_1} \ar[rd]^{p_2}   \\
		& \Gr(k,\CC^N)  && \Gr(k-1,\CC^N)
	}
\end{equation}  where $\Fl(k-1,k)$ is the 3-step partial flag variety and $p_{1}$, $p_{2}$ are the natural projections. Then we define 
${\E}:=p_{2*}p^{*}_{1}$ and a similar functor ${\F}$ in the opposite direction that can be viewed as lift of $e$ and $f$, respectively.  The functors ${\E}$, ${\F}$  satisfy the defining relations of $\Uu_{q}(\SL_{2})$, i.e., we have 
\begin{equation*} 
	{\E\F}|_{\Cc(\lambda)} \cong {\F\E}|_{\Cc(\lambda)} \bigoplus \bigoplus_{[\lambda] } \text{Id}_{\Cc(\lambda)} \ \text{if} \ \lambda \geq 0,
\end{equation*} similarly for $\lambda \leq 0$. Here $[n]:=q^{-n+1}+q^{-n+3}+....+q^{n-3}+q^{n-1}$ is the quantum integer and $\bigoplus_{[\lambda] }(.)$ denotes a graded direct sum. For example, $\bigoplus_{[2] } f =f[1] \bigoplus f[-1]$, where $[1]$ is the homological shift in the derived category that uses to keep track of the power of the variable $q$.

Motivated by the above construction, we now consider the weight categories to be $\Dd^b(\Gr(k,\CC^N))$, i.e., we replace the constructible sheaves by coherent sheaves. Let $\V, \ \V'$ to be the tautological bundles on $Fl(k-1,k)$ of rank $k$, $k-1$ respectively. Then there is a natural line bundle $\V/\V'$ on $Fl(k-1,k)$. Using the same correspondence (\ref{diag 1}), instead of just pulling back and pushing forward directly, we have an extra twist by the line bundles $(\V/\V')^{r}$ where $r \in \ZZ$. So we have the functors
\begin{equation*} 
	{\E}_{r}:=p_{2*}(p_{1}^{*}\otimes(\V/\V')^{r}):\Dd^{b}(\Gr(k,\CC^N)) \rightarrow \Dd^{b}(\Gr(k-1,\CC^N))
\end{equation*} and similarly ${\F}_{r}$ where $r \in \ZZ$.

The shifted $q=0$ affine algebra arises naturally from the studies of the  $L\SL_{2}:=\SL_{2} \otimes \CC[t,t^{-1}]$-like algebra acting on $\bigoplus_{k} \Dd^{b}(\Gr(k,\CC^N))$, where $e \otimes t^{r}$ and $f \otimes t^{s}$ acting via the functors ${\E}_{r}$, ${\F}_{s}$ respectively for $r, \ s \in \ZZ$. In \cite{Hsu}, the author is the first one to study this action in detail.  

After decategorifying (pass to the K-theory), we obtain an algebra with loop generators $e_{r}$, $f_{s}$ and relations that look similar to the shifted quantum affine algebra defined in \cite{FT}. There is a variable $q$ in their definition that stands for the $q$-deformation of quantum affine algebra. From the geometric aspect, it comes from the $\CC^*$-action. In our case, we do not have a natural $\CC^*$-action on Grassmannians, so we do not have a variable like $q$. 

We call the resulting algebra shifted $q=0$ affine algebra. First, the "$q=0$" part in the name comes from the reason that some of the relations can be obtained from the relations in shifted quantum affine algebra by taking $q=0$ directly. Second, the "shifted" comes from the fact that the commutator relations between $e_{r}$ and $f_{s}$ vanish within a certain range of $r+s$. Finally, even though the algebra is given by the loop generators $e_{r}$, $f_{s}$,  the representation $\bigoplus_{k}K(\Gr(k,\CC^N))$ we study is finite-dimensional. Since affine algebras are central extensions of loop algebras, and central extension acts trivially on the finite-dimensional representations, we still use "affine algebra" in the name.

\subsection{Shifted $q=0$ affine algebras}

In this section, we define the shifted $q=0$ affine algebras. By imitating the presentation in \cite{FT} (the so-called Levedorskii type presentation), the presentation we use here is by finite numbers of generators and defining relations. In \cite{Hsu}, we also give a conjectural presentation defined by generating series and conjecture that the two presentations are equivalent; see conjecture A.2. in  \textit{loc. cit.}.

Similarly to the dot version $\dot{\Uu}_{q}(\SL_2)$ of $\Uu_{q}(\SL_2)$ that introduced in \cite{BLM}, the shifted $q=0$ affine algebras we introduce below is also an idempotent version. This means that we replace the identity with the direct sum of a system of projectors, one for each element of the weight lattices. They are orthogonal idempotents for approximating the unit element. We refer to part IV in \cite{Lu} for details of such modification.

Throughout the rest of this article, we fix a positive integer $N \geq 2$. Let
\[
C(n,N):=\{\underline{k}=(k_1,...,k_{n}) \in \NN^n\ | \ k_{1}+...+k_{n}=N \}.
\] 

We regard each $\underline{k}$ as a weight for $\SL_{n}$ via the identification of the weight lattice of $\SL_n$ with the quotient $\ZZ^n/(1,1,...,1)$. We choose the simple root $\alpha_{i}$ to be $(0...0,-1,1,0...0)$ where the $-1$ is in the $i$-th position for $1 \leq i \leq n-1$. Finally, we denote $\langle\cdot,\cdot\rangle:\ZZ^n \times \ZZ^n \rightarrow \ZZ$ to be the standard pairing.

\begin{definition} \label{definition 1}
The \textit{shifted $q=0$ affine algebras}, denoted by $\dot{\Uu}_{0,N}(L\SL_n)$, is the associative $\CC$-algebra generated by
\begin{equation*}
\{1_{\kk}, \ e_{i,r}1_{\kk}, \ f_{i,s}1_{\kk},\ (\psi^{+}_{i})^{\pm 1}1_{\kk}, \ (\psi^{-}_{i})^{\pm 1}1_{\kk},\ h_{i,\pm 1}1_{\kk} \ | \ \kk \in C(n,N), \ 1 \leq i \leq n-1, \ -k_{i}-1 \leq r \leq 0, \ 0 \leq s \leq k_{i+1}+1 \}
\end{equation*}
	with the following relations
	\begin{equation} \tag{U01}   \label{U01}
		1_{\kk}1_{\Ll}=\delta_{\kk,\Ll}1_{\kk}, \ e_{i,r}1_{\kk}=1_{\kk+\alpha_{i}}e_{i,r}, \   f_{i,r}1_{\kk}=1_{\kk-\alpha_{i}}f_{i,r}, \ (\psi^{+}_{i})^{\pm 1}1_{\kk}=1_{\kk}(\psi^{+}_{i})^{\pm 1},\
		h_{i,\pm 1}1_{\kk}=1_{\kk}h_{i, \pm 1},	
	\end{equation}
	\begin{equation}\tag{U02}   \label{U02}
		\{(\psi^{+}_{i})^{\pm 1}1_{\kk},(\psi^{-}_{i})^{\pm 1}1_{\kk},h_{i,\pm 1}1_{\kk}\ | \ 1\leq i \leq n-1,\kk \in C(n,N)\} \ \mathrm{pairwise\ commute},
	\end{equation}
	\begin{equation} \tag{U03} \label{U03}
		(\psi^{+}_{i})^{\pm 1} \cdot (\psi^{+}_{i})^{\mp 1} 1_{\kk} = 1_{\kk}=(\psi^{-}_{i})^{\pm 1} \cdot (\psi^{-}_{i})^{\mp 1}1_{\kk},
	\end{equation}
	\begin{align*} \tag{U04} \label{U04}
		&e_{i,r}e_{j,s}1_{\kk}=\begin{cases}
			-e_{i,s+1}e_{i,r-1}1_{\kk} & \text{if} \ j=i \\
			e_{i+1,s}e_{i,r}1_{\kk}-e_{i+1,s-1}e_{i,r+1}1_{\kk} & \text{if} \ j=i+1 \\
			e_{i,r+1}e_{i-1,s-1}1_{\kk}-e_{i-1,s-1}e_{i,r+1}1_{\kk} &\text{if} \ j=i-1 \\
			e_{j,s}e_{i,r}1_{\kk} &\text{if} \ |i-j| \geq 2
		\end{cases},
	\end{align*}
	
	\begin{align*}\tag{U05} \label{U05}
		&f_{i,r}f_{j,s}1_{\kk}=\begin{cases}
			-f_{i,s-1}f_{i,r+1}1_{\kk} & \text{if} \ j=i \\
			f_{i,r-1}f_{i+1,s+1}1_{\kk}-f_{i+1,s+1}f_{i,r-1}1_{\kk} & \text{if} \ j=i+1 \\
			f_{i-1,s}f_{i,r}1_{\kk}-f_{i-1,s+1}f_{i,r-1}1_{\kk} &\text{if} \ j=i-1 \\
			f_{j,s}f_{i,r}1_{\kk} &\text{if} \ |i-j| \geq 2
		\end{cases},
	\end{align*}
	
	
	\begin{align*} \tag{U06} \label{U06}
		&\psi^{+}_{i}e_{j,r}1_{\kk}=
		\begin{cases}
			-e_{i,r+1}\psi^{+}_{i}1_{\kk} & \text{if} \ j=i \\
			-e_{i+1,r-1}\psi^{+}_{i}1_{\kk} & \text{if} \ j=i+1 \\
			e_{i-1,r}\psi^{+}_{i}1_{\kk}  &  \text{if} \ j=i-1 \\
			e_{j,r}\psi^{+}_{i}1_{\kk}  &  \text{if} \ |i-j| \geq 2 \\
		\end{cases},
		&\psi^{-}_{i}e_{j,r}1_{\kk}=
		\begin{cases}
			-e_{i,r+1}\psi^{-}_{i}1_{\kk} & \text{if} \ j=i \\
			e_{i+1,r}\psi^{-}_{i}1_{\kk} & \text{if} \ j=i+1 \\
			-e_{i-1,r-1}\psi^{-}_{i}1_{\kk}  &  \text{if} \ j=i-1 \\
			e_{j,r}\psi^{-}_{i}1_{\kk}  &  \text{if} \ |i-j| \geq 2 \\
		\end{cases},
	\end{align*}
	\begin{align*} \tag{U07} \label{U07}
		&\psi^{+}_{i}f_{j,r}1_{\kk}=
		\begin{cases}
			-f_{i,r-1}\psi^{+}_{i}1_{\kk} & \text{if} \ j=i \\
			-f_{i+1,r+1}\psi^{+}_{i}1_{\kk} & \text{if} \ j=i+1 \\
			f_{i-1,r}\psi^{+}_{i}1_{\kk}  &  \text{if} \ j=i-1 \\
			f_{j,r}\psi^{+}_{i}1_{\kk}  &  \text{if} \ |i-j| \geq 2 \\
		\end{cases},
		&\psi^{-}_{i}f_{j,r}1_{\kk}=
		\begin{cases}
			-f_{i,r-1}\psi^{-}_{i}1_{\kk} & \text{if} \ j=i \\
			f_{i+1,r}\psi^{-}_{i}1_{\kk} & \text{if} \ j=i+1 \\
			-f_{i-1,r+1}\psi^{-}_{i}1_{\kk}  &  \text{if} \ j=i-1 \\
			f_{j,r}\psi^{-}_{i}1_{\kk}  &  \text{if} \ |i-j| \geq 2 \\
		\end{cases},
	\end{align*}
	
	
	\begin{align*} \tag{U08} \label{U08} 
		&[h_{i,\pm 1},e_{j,r}]1_{\kk}=\begin{cases}
			0 & \text{if} \ i=j \\
			-e_{i+1,r\pm 1}1_{\kk} & \text{if} \ j=i+1 \\
			e_{i-1,r\pm 1}1_{\kk} & \text{if} \  j=i-1 \\
			0& \text{if} \ |i-j| \geq 2 \\
		\end{cases}, 
		&[h_{i,\pm 1},f_{j,r}]1_{\kk}=\begin{cases}
			0 & \text{if} \ i=j \\
			f_{i+1,r\pm 1}1_{\kk} & \text{if} \ j=i+1 \\
			-f_{i-1,r\pm 1}1_{\kk} & \text{if} \  j=i-1 \\
			0& \text{if} \ |i-j| \geq 2 \\
		\end{cases},
	\end{align*}
	
	\begin{equation} \tag{U09} \label{U09} 
		[e_{i,r},f_{j,s}]1_{\kk}=0 \ \mathrm{if} \ i \neq j\ \ \mathrm{and}\ \  [e_{i,r},f_{i,s}]1_{\kk}= \begin{cases}
			\psi^{+}_{i}h_{i,1}1_{\kk} & \text{if} \  r+s=k_{i+1}+1 \\
			\psi^{+}_{i}1_{\kk} & \text{if} \ r+s=k_{i+1} \\
			0 & \text{if} \  -k_{i}+1 \leq r+s \leq k_{i+1}-1 \\
			-\psi^{-}_{i}1_{\kk} & \text{if} \ r+s=-k_{i} \\
			-\psi^{-}_{i}h_{i,-1}1_{\kk} & \text{if} \ r+s=-k_{i}-1
		\end{cases},
	\end{equation}
	
	for any $1 \leq i,j \leq n-1$ and $r,s$ such that the above relations make sense.
\end{definition}

\subsection{Categorical $\dot{\Uu}_{0,N}(L\SL_n)$ action}

In this section, we recall the definition of the categorical action for shifted $q=0$ affine algebra that is defined in \cite{Hsu}. However, since most of the relations in the categorical action will not be used for the proof of our results, we only list those that will be used and call such action a \textit{partial categorical action}. 
We refer the readers to Definition 3.1. in \textit{loc. cit.} for a full definition of the categorical action.

\begin{definition} \label{definition 2}
	A \textit{partial categorical $\dot{\Uu}_{0,N}(L\SL_n)$ action} consists of a target 2-category $\Kk$, which is triangulated, $\CC$-linear and idempotent complete. The objects in $\Kk$ are
	\[
	\mathrm{Ob}(\Kk)=\{\Kk(\kk)\ |\ \kk \in C(n,N) \}
	\] where each $\Kk(\kk)$ is a triangulated category, and  $\mathrm{Hom}(\Kk(\kk),\Kk(\Ll))$ is also a triangulated category for all $\kk, \Ll \in C(n,N)$.  On those objects we equip with the following 1-morphisms $\bo_{\kk}$, ${\E}_{i,r}\bo_{\kk}=\bo_{\kk+\alpha_{i}}{\E}_{i,r}$, ${\F}_{i,s}\bo_{\kk}=\bo_{\kk-\alpha_{i}}{\F}_{i,s}$, $({\spi}^{\pm}_{i})^{\pm 1}\bo_{\kk}=\bo_{\kk}({\spi}^{\pm}_{i})^{\pm 1}$, ${\Hhh}_{i,\pm1}\bo_{\kk}=\bo_{\kk}{\Hhh}_{i,\pm 1}$   where $1 \leq i \leq n-1$, $-k_{i}-1 \leq r \leq 0$, $0 \leq s \leq k_{i+1}+1$. Here $\bo_{\kk}$ is the identity 1-morphism of $\Kk(\kk)$.
	
	On this data, we impose the following conditions
	
	\begin{enumerate}
		\item The space of maps between any two 1-morphisms is finite-dimensional.
		\item If $\alpha=\alpha_{i}$ or $\alpha=\alpha_{i}+\alpha_{j}$ for some $i,j$ with $\langle\alpha_{i},\alpha_{j}\rangle=-1$, then $\bo_{\kk+r\alpha}=0$ for $r \gg 0$ or  $r \ll 0$.
		\item Suppose $i \neq j$. If $\bo_{\kk+\alpha_{i}}$ and $\bo_{\kk+\alpha_{j}}$ are nonzero, then $\bo_{\kk}$ and $\bo_{\kk+\alpha_{i}+\alpha_{j}}$ are also nonzero.
		\item The right adjoint of ${\E}_{i,r}$ and ${\F}_{i,s}$ are given by conjugation of $\spi^{\pm}_{i}$ up to homological shifts. More precisely, 
		\begin{enumerate}
		\item $({\E}_{i,r}\bo_{\kk})^{R} \cong \bo_{\kk}{({\spi}^{+}_{i})^{r+1}}{\F}_{i,k_{i+1}+2}({\spi}^{+}_{i})^{-r-2}[-r-1]$ for all $1 \leq i \leq n-1$,
		\item $({\F}_{i,s}\bo_{\kk})^{R} \cong \bo_{\kk}({\spi}^{-}_{i})^{-s+1}{\E}_{i,-k_{i}-2}({\spi}^{-}_{i})^{s-2}[s-1]$ for all $1 \leq i \leq n-1$.
	\end{enumerate}
		
		\item The relations between ${\E}_{i,r}$, ${\F}_{i,s}$, $\Psi^{\pm}_{i}$ are given by the following
			\begin{enumerate}
				\item ${\spi}^{\pm}_{i}{\E}_{i,r}\bo_{\kk} \cong {\E}_{i,r+1}{\spi}^{\pm}_{i}\bo_{\kk}[\mp 1]$.
				\item ${\spi}^{\pm}_{i}{\F}_{i,s}\bo_{\kk} \cong {\F}_{i,s-1}{\spi}^{\pm}_{i}\bo_{\kk}[\pm 1]$.
			\end{enumerate}
	
		\item For ${\E}_{i,r}{\F}_{i,s}\bo_{\kk}, {\F}_{i,s}{\E}_{i,r}\bo_{\kk} \in \mathrm{Hom}(\Kk(\kk),\Kk(\kk))$, they are related by  exact triangles, more precisely, 
		\begin{enumerate}
			\item \[
			{\F}_{i,s}{\E}_{i,r}\bo_{\kk} \rightarrow {\E}_{i,r}{\F}_{i,s}\bo_{\kk} \rightarrow {\spi}^{+}_{i}\bo_{\kk} 
			\ \text{if} \  
			r+s =k_{i+1},
			\] 
			\item \[
			{\E}_{i,r}{\F}_{i,s}\bo_{\kk} \rightarrow {\F}_{i,s}{\E}_{i,r}\bo_{\kk} \rightarrow {\spi}^{-}_{i}\bo_{\kk}   \ \text{if} \  r+s = -k_{i},
			\] 
			\item \[
			{\F}_{i,s}{\E}_{i,r}\bo_{\kk} \cong {\E}_{i,r}{\F}_{i,s}\bo_{\kk} \ \text{if} \ -k_{i}+1 \leq r+s \leq k_{i+1}-1 .
			\] 
		\end{enumerate}
	\end{enumerate}
	for all $r$, $s$ that make the above conditions make sense, and the isomorphisms between functors that appear in every condition are abstractly defined, i.e., we do not specify any 2-morphisms that induce those isomorphisms. Also, we do not require any conditions on ${\Hhh}_{i,\pm1}\bo_{\kk}$.
\end{definition}

\subsection{Geometric example}

In this section, we mention the partial main result, i.e., Theorem 5.2 in \cite{Hsu}, which says that there is a  \textit{partial categorical action} of $\dot{\Uu}_{0,N}(L\SL_n)$ on $\bigoplus_{\kk}\Dd^b(\Fl_{\kk}(\CC^N))$.

We denote $Y(\kk)=\Fl_{\kk}(\CC^N)$ and $\Dd^b(Y(\kk))$ will be the objects $\Kk(\kk)$ of the triangulated $2$-category $\Kk$ in Definition \ref{definition 2}. To define those $1$-morphisms  ${\E}_{i,r}\bo_{\kk}$, ${\F}_{i,s}\bo_{\kk}$, ${\Hhh}_{i,\pm 1}\bo_{\kk}$, $({\spi}^{\pm}_{i})^{\pm 1}\bo_{\kk}$, we use the language of FM transforms, that means we will define them by using FM kernels. However, since the FM kernels for ${\Hhh}_{i,\pm 1}\bo_{\kk}$ will not be used in the rest of this article, we only mention the FM kernels for ${\E}_{i,r}\bo_{\kk}$, ${\F}_{i,s}\bo_{\kk}$, $({\spi}^{\pm}_{i})^{\pm 1}\bo_{\kk}$ for simplicity.

We define correspondences $W^{1}_{i}(\kk) \subset Y(\kk) \times Y(\kk+\alpha_{i})$ by
\begin{equation*} 
	W^{1}_{i}(\kk):=\{(V_{\bullet},V_{\bullet}') \in  Y(\kk) \times Y(\kk+\alpha_{i}) \ | \ V_{j}=V'_{j} \ \Rm{for} \ j \neq i,  \ V'_{i} \subset V_{i}\}, 
\end{equation*} then we have the natural line bundle $\V_{i}/\V'_{i}$ on $W^{1}_{i}(\kk)$ where $\V_{i}$ and $\V'_{i}$ are the tautological bundles on $W^{1}_{i}(\kk)$ whose fibres over a point $(V_{\bullet},V_{\bullet}')$ equal to $V_{i}$ and $V'_{i}$, respectively.

We also have the transpose correspondence $^{\TTt}W^{1}_{i}(\kk) \subset Y(\kk+\alpha_{i}) \times Y(\kk)$. Let $\iota(\kk):W^{1}_{i}(\kk) \hookrightarrow Y(\kk) \times Y(\kk+\alpha_{i})$, $^{\TTt}\iota(\kk):^{\TTt}W^{1}_{i}(\kk) \hookrightarrow Y(\kk+\alpha_{i}) \times Y(\kk)$ be the inclusions, and $\Delta(\kk):Y(\kk) \rightarrow Y(\kk) \times Y(\kk)$ be the diagonal map. Then we have the following theorem.

\begin{theorem} [Theorem 5.2 \cite{Hsu}] \label{Theorem 3}
Let $\Kk$ be the triangulated 2-categories whose nonzero objects are $\Kk(\kk)=\Dd^b(Y(\kk))$ where $\kk \in C(n,N)$. The 1-morphisms ${\E}_{i,r}\bo_{\kk}$, $\bo_{\kk}{\F}_{i,s}$, $({\spi}^{\pm}_{i})^{\pm 1}\bo_{\kk}$ are FM transformations with kernels given by $\Ee_{i,r}\bo_{\kk}:=\iota(\kk)_{*} (\V_{i}/\V_{i}')^{r}$,
$\bo_{\kk}\Ff_{i,s}:=^{\TTt}\iota(\kk)_{*} (\V_{i}/\V'_{i})^{s}$, $(\Psi^{+}_{i})^{\pm 1}\bo_{\kk}:=\Delta(\kk)_{*}\tdet (\V_{i+1}/\V_{i})^{\pm 1}[\pm(1-k_{i+1})]$, $(\Psi^{-}_{i})^{\pm 1}\bo_{\kk}:=\Delta(\kk)_{*} \tdet (\V_{i}/\V_{i-1})^{\mp 1}[\pm(1-k_{i})]$ respectively, and the 2-morphisms are maps between kernels.  Then this gives a partial categorical $\dot{\Uu}_{0,N}(L\SL_n)$ action.
\end{theorem}

\section{Main results}

In this section, we prove the main result of this article and state the consequence. Roughly speaking, it says that certain functors in the categorical action of $\dot{\Uu}_{0,N}(L\SL_{n})$ satisfy the semiorthogonal property, and thus give rise to $\SOD$ of each weight category. 

\subsection{Main theorem}

The following is the main result of this article.
\begin{theorem} \label{Proposition 4}
	Given a partial categorical $\dot{\Uu}_{0,N}(L\SL_n)$ action $\Kk$. Considering the functors
	\begin{equation*}
		{\F}_{1,\blam(1)}{\F}_{2,\blam(2)} ... {\F}_{n-1, \blam(n-1)} \bo_{\eta} \in \Hom(\Kk(\eta),\Kk(\kk))
	\end{equation*}  where ${\F}_{i,\blam(i)}:= {\F}_{i, \lambda(i)_1} {\F}_{i, \lambda(i)_2} ...  {\F}_{i, \lambda(i)_{\Bbbk_{i}}}$ with $\blam(i)=(\lambda(i)_{1},...,\lambda(i)_{\Bbbk_{i}}) \in P(k_{i+1},\Bbbk_{i})$, $\Bbbk_{i}=\sum_{j=1}^{i}k_{j}$ for all $1 \leq i \leq n-1$, and $\eta=(0,0,...,0,N)$ be the highest weight. Then they satisfy the following properties 
	\begin{enumerate}
		\item $\Hom({\F}_{1,\blam(1)}{\F}_{2,\blam(2)} ... {\F}_{n-1, \blam(n-1)} \bo_{\eta},{\F}_{1,\blam(1)}{\F}_{2,\blam(2)} ... {\F}_{n-1, \blam(n-1)}\bo_{\eta}) \cong \Hom(\bo_{\eta},\bo_{\eta})$,
		\item   $\Hom({\F}_{1,\blam(1)}{\F}_{2,\blam(2)} ... {\F}_{n-1, \blam(n-1)} \bo_{\eta},{\F}_{1,\blam(1)'}{\F}_{2,\blam(2)'} ... {\F}_{n-1, \blam(n-1)'}\bo_{\eta}) \cong 0$ if $(\blam(1),...,\blam(n-1)) <_{pl}  (\blam(1)',...,\blam(n-1)')$ where $<_{pl}$ denotes the product lexicographic order; i.e., there exists $1 \leq i \leq n-1$ such that $\blam(j)=\blam(j)'$ for all $0 \leq j \leq i-1$ and $\blam(i) <_{l} \blam(i)'$.
	\end{enumerate}
\end{theorem}

Note that the first property is not the same as the definition of \textit{relative exceptional objects} defined in \cite{BLMHAP} (see Remark \ref{remark 2} for an explanation), and the second property implies that the collection of such functors forms a semiorthogonal collection.

Moreover, there are two direct consequences that are easy to see. The first one is each functor ${\F}_{1,\blam(1)}{\F}_{2,\blam(2)} ... {\F}_{n-1, \blam(n-1)} \bo_{\eta}$ is fully-faithful. The second one is the subcategories that are generated by the essential image of the collection of such functors giving rise to a $\SOD$ of the weight category $\Kk(\kk)$. 

Instead of stating these (abstract) corollaries first, we decide to let the readers know some concrete examples. By Theorem \ref{Theorem 3} there is a partial categorical $\dot{\Uu}_{0,N}(L\SL_n)$ action on the derived categories of coherent sheaves on $n$-step partial flag varieties. 

So now the weight categories are $\Kk(\kk)=\Dd^b(\Fl_{\kk}(\CC^N))$. We also know that each functor ${\F}_{i,s}$ is given by an FM kernel $\Ff_{i,s}$. By Proposition \ref{Proposition 2}, the composition of functors is given by the convolution of FM kernels. Considering $\blam(i)=(\blam(i)_{1},...,\blam(i)_{\Bbbk_{i}}) \in P(k_{i+1},\Bbbk_{i})$ for all $1 \leq i \leq n-1$, then we obtain
\begin{equation} 
	\bigotimes_{i=1}^{n-1} \s_{\blam(i)}\V_{i}
	\cong \Ff_{1,\blam(1)} \ast \Ff_{2,\blam(2)} \ast ... \ast \Ff_{n-1, \blam(n-1)} \bo_{\eta}	 \in \Dd^b(\Fl_{\eta}(\CC^N)\times \Fl_{\kk}(\CC^N)),  \label{eq 3}
\end{equation} where $\Ff_{i,\blam(i)} := \Ff_{i, \blam(i)_{1}} \ast \Ff_{i, \blam(i)_{2}} ... \ast \Ff_{i, \blam(i)_{\Bbbk_{i}}}$ and  $\eta:=(0,0,...,0,N)$ is the highest weight.  

Note that when $n=2$ and $(k_1,k_2)=(k,N-k)$, (\ref{eq 3}) becomes (\ref{eq 16}) in Subsubsection \ref{subsub1.3} which is the Kapranov exceptional collection for Grassmannians $\Gr(k,\CC^N)$
\begin{equation*}
	\s_{\blam}\V \cong \Ff_{\lambda_{1}} \ast ... \ast \Ff_{\lambda_{k}} \bo_{(0,N)}  \in \Dd^b(\Gr(0,\CC^N)\times \Gr(k,\CC^N)).
\end{equation*}



From Theorem \ref{Theorem 1}, \ref{Theorem 2} we know that $\{\s_{\lambda}\V \}$ and $\{ \bigotimes_{i=1}^{n-1} \s_{\lambda_{i}}\V_{i} \}$ are full exceptional collections for $\Dd^b(\Gr(k,\CC^N))$ and $\Dd^b(\Fl_{\kk}(\CC^N))$ respectively. Hence Theorem \ref{Proposition 4} recovers the classical Kapranov exceptional collection.


\subsection{Proof of the main theorem}

We prove Theorem \ref{Proposition 4} in this section. Since the result is about $\Hom$ space, the idea of the proof is taking the right adjoint of ${\F}_{1,\blam(1)}{\F}_{2,\blam(2)} ... {\F}_{n-1, \blam(n-1)} \bo_{\eta}$ and calculate compositions with other functors.  More explicitly, we calculate 
\begin{align} \label{adcom}
	\begin{split}
	&({\F}_{1,\blam(1)}{\F}_{2,\blam(2)} ... {\F}_{n-1, \blam(n-1)})^{R}{\F}_{1,\blam(1)'}{\F}_{2,\blam(2)'} ... {\F}_{n-1, \blam(n-1)'} \bo_{\eta}  \\
	& \cong ({\F}_{n-1, \blam(n-1)})^{R}...({\F}_{2,\blam(2)})^{R}({\F}_{1,\blam(1)})^{R}{\F}_{1,\blam(1)'}{\F}_{2,\blam(2)'} ... {\F}_{n-1, \blam(n-1)'} \bo_{\eta}
	\end{split}
\end{align} for all $\blam(i), \ \blam(i)'  \in P(k_{i+1},\Bbbk_{i})$, $1 \leq i \leq n-1$. We show that $(\ref{adcom}) \cong \bo_{\eta}$ if $(\blam(1),...,\blam(n-1))=(\blam(1)',...,\blam(n-1)')$, and $(\ref{adcom}) \cong 0$ if $(\blam(1),...,\blam(n-1)) <_{pl} (\blam(1)',...,\blam(n-1)')$. Then the result follows directly.

The first step we have to calculate is the composition $({\F}_{1,\blam(1)})^{R}{\F}_{1,\blam(1)'}\bo_{(0,k_{1}+k_{2},k_{3},...,k_{n})}$ where $\blam(1), \blam(1)' \in P(k_{2},k_{1})$. Note that the two functors ${\F}_{1,\blam(1)}$, ${\F}_{1,\blam(1)'}$ have the same sub-index, so the calculation only depends on the first two terms in the weight $(0,k_{1}+k_{2},k_{3},...,k_{n})$. Thus it suffices to prove Theorem \ref{Proposition 4} for the $\SL_{2}$ version, which is the following.

\begin{theorem} \label{Proposition 3}
Given a partial categorical $\dot{\Uu}_{0,N}(L\SL_2)$ action $\Kk$. Considering the functors 
\begin{equation*}
	{\F}_{\blam}\bo_{(0,N)}:={\F}_{\lambda_{1}}{\F}_{\lambda_{2}} ... {\F}_{\lambda_{k}} \bo_{(0,N)} \in \Hom(\Kk(0,N),\Kk(k,N-k))	
\end{equation*} where $\blam=(\lambda_{1},...,\lambda_{k}) \in P(N-k,k)$. Then $\{{\F}_{\blam}\bo_{(0,N)}\}_{\blam \in P(N-k,k)}$ satisfy the following properties
\begin{enumerate}
\item  $\Hom({\F}_{\blam}\bo_{(0,N)},{\F}_{\blam}\bo_{(0,N)}) \cong \Hom(\bo_{(0,N)},\bo_{(0,N)})$,
\item $\Hom({\F}_{\blam}\bo_{(0,N)},{\F}_{\blam'}\bo_{(0,N)}) \cong 0$ 
if $\blam=(\lambda_{1},...,\lambda_{k})<_{l} \blam'=(\lambda'_{1},...,\lambda'_{k})$.
\end{enumerate}
\end{theorem}

We dedicate the rest of this section to the proof of Theorem \ref{Proposition 3}. 

\begin{proof}[Proof of Theorem \ref{Proposition 3}]
In this proof, we will keep using the categorical relations in Definition \ref{definition 2}, in particular, relation (4)(b), relation (5)(b) and relation (6)(b)(c).

We prove the first property. Given a functor ${\F}_{\blam}\bo_{(0,N)}:={\F}_{\lambda_{1}}{\F}_{\lambda_{2}} ... {\F}_{\lambda_{k}} \bo_{(0,N)} \in \Hom(\Kk(0,N),\Kk(k,N-k))$ where $\blam=(\lambda_{1},...,\lambda_{k}) \in P(N-k,k)$. Applying the adjunction to get 
\begin{equation*}				  \Hom({\F}_{\blam}\bo_{(0,N)},{\F}_{\blam}\bo_{(0,N)})=\Hom(\bo_{(0,N)},({\F}_{\blam}\bo_{(0,N)})^{R}{\F}_{\blam}\bo_{(0,N)}).
\end{equation*}

So we have to know the right adjoint functor first. We formulate it in the following lemma where the proof is simply applying relation (4)(b) in Definition \ref{definition 2} and we leave it to the reader.

\begin{lemma} \label{ra}
 The right adjoint is given by	
	\begin{equation*}
	({\F}_{\blam}\bo_{(0,N)})^{R} =({\spi^{-}})^{-\lambda_{k}+1}{\E}_{-2}({\spi^{-}})^{\lambda_{k}-\lambda_{k-1}-1}{\E}_{-3}({\spi^{-}})^{\lambda_{k-1}-\lambda_{k-2}-1}...{\E}_{-k-1}({\spi^{-}})^{\lambda_{1}-2}\bo_{(k,N-k)}[\sum_{i=1}^{k}\lambda_{i}-k].
	\end{equation*}
\end{lemma}
 
	
By Lemma \ref{ra}, to calculate $({\F}_{\lambda}\bo_{(0,N)})^{R}{\F}_{\lambda}\bo_{(0,N)}$, we need to simplify 
\begin{equation} \label{eq 5}
({\spi^{-}})^{-\lambda_{k}+1}{\E}_{-2}({\spi^{-}})^{\lambda_{k}-\lambda_{k-1}-1}{\E}_{-3}({\spi^{-}})^{\lambda_{k-1}-\lambda_{k-2}-1}...{\E}_{-k-1}({\spi^{-}})^{\lambda_{1}-2}{\F}_{\lambda_{1}}{\F}_{\lambda_{2}} ... {\F}_{\lambda_{k}} \bo_{(0,N)}[\sum_{i=1}^{k}\lambda_{i}-k].
\end{equation}

We will keep using the following lemma to simplify (\ref{eq 5}).

\begin{lemma} \label{simplem}
We have 
\begin{equation*}
({\spi^{-}})^{-\lambda_{k}+1}{\E}_{-2}({\spi^{-}})^{\lambda_{k}-\lambda_{k-1}-1}{\E}_{-3}...{\E}_{-k+i-2}({\spi^{-}})^{\lambda_{i}-\lambda_{i-1}-1}{\F}_{2}\bo_{(k-i,N-k+i)} \cong 0
\end{equation*}	for all $2 \leq i \leq k$.
\end{lemma}
	
\begin{proof}
By relation (5)(b), we have 
\begin{equation*}
	{\E}_{-k+i-2}({\spi^{-}})^{\lambda_{i}-\lambda_{i-1}-1}{\F}_{2}\bo_{(k-i,N-k+i)} \cong {\E}_{-k+i-2}{\F}_{\lambda_{i-1}-\lambda_{i}+3}({\spi^{-}})^{\lambda_{i}-\lambda_{i-1}-1}\bo_{(k-i,N-k+i)}[-\lambda_{i}+\lambda_{i-1}+1] 
\end{equation*}

Since $\blam=(\lambda_{1},...,\lambda_{k}) \in P(N-k,k)$, we have $0 \leq \lambda_{i-1}-\lambda_{i} \leq N-k$ and $k-i \geq 0$  also implies that $2 \leq k$. Thus we obtain 
\begin{equation*}
-k+i+1 \leq \lambda_{i-1}-\lambda_{i}-k+i+1 \leq N-2k+i+1 \leq N-k+i-1.
\end{equation*} 

By relation (6)(c), we get 
	\begin{align*}
	&({\spi^{-}})^{-\lambda_{k}+1}{\E}_{-2}({\spi^{-}})^{\lambda_{k}-\lambda_{k-1}-1}{\E}_{-3}...{\E}_{-k+i-2}({\spi^{-}})^{\lambda_{i}-\lambda_{i-1}-1}{\F}_{2}\bo_{(k-i,N-k+i)}  \\
	&\cong ({\spi^{-}})^{-\lambda_{k}+1}{\E}_{-2}({\spi^{-}})^{\lambda_{k}-\lambda_{k-1}-1}{\E}_{-3}...{\F}_{\lambda_{i-1}-\lambda_{i}+3}{\E}_{-k+i-2}({\spi^{-}})^{\lambda_{i}-\lambda_{i-1}-1}\bo_{(k-i,N-k+i)}[-\lambda_{i}+\lambda_{i-1}+1].
\end{align*}

The next term we have to calculate is ${\E}_{-k+i-1}({\spi^{-}})^{\lambda_{i+1}-\lambda_{i}-1}{\F}_{\lambda_{i-1}-\lambda_{i}+3}\bo_{(k-i-1,N-k+i+1)}$ and similarly 
\begin{align*}
&{\E}_{-k+i-1}({\spi^{-}})^{\lambda_{i+1}-\lambda_{i}-1}{\F}_{\lambda_{i-1}-\lambda_{i}+3}\bo_{(k-i-1,N-k+i+1)} \\
&\cong {\E}_{-k+i-1}{\F}_{\lambda_{i-1}-\lambda_{i+1}+4}({\spi^{-}})^{\lambda_{i+1}-\lambda_{i}-1}\bo_{(k-i-1,N-k+i+1)}[-\lambda_{i+1}+\lambda_{i}+1].
\end{align*} Again, we use $0 \leq \lambda_{i-1}-\lambda_{i+1} \leq N-k$ and $k-i-1 \geq 0$ implies $k \geq 3$. Thus we get
\begin{equation*}
	-k+i+3 \leq \lambda_{i-1}-\lambda_{i+1}-k+i+3 \leq N-2k+i+3 \leq N-k+i,
\end{equation*} and by relation (6)(c) again we get
\begin{equation*}
{\E}_{-k+i-1}{\F}_{\lambda_{i-1}-\lambda_{i+1}+4}\bo_{(k-i-1,N-k+i+1)} \cong {\F}_{\lambda_{i-1}-\lambda_{i+1}+4}{\E}_{-k+i-1}\bo_{(k-i-1,N-k+i+1)}.	
\end{equation*} 

Continuing this process and ignoring the homological shifts, we have 
\begin{align*}
	&({\spi^{-}})^{-\lambda_{k}+1}{\E}_{-2}({\spi^{-}})^{\lambda_{k}-\lambda_{k-1}-1}{\E}_{-3}...{\E}_{-k+i-2}({\spi^{-}})^{\lambda_{i}-\lambda_{i-1}-1}{\F}_{2}\bo_{(k-i,N-k+i)} \\
	& \cong ({\spi^{-}})^{-\lambda_{k}+1}{\E}_{-2}({\spi^{-}})^{\lambda_{k}-\lambda_{k-1}-1}{\E}_{-3}...{\F}_{\lambda_{i-1}-\lambda_{i}+3}{\E}_{-k+i-2}({\spi^{-}})^{\lambda_{i}-\lambda_{i-1}-1}\bo_{(k-i,N-k+i)} \\
	&\cong ({\spi^{-}})^{-\lambda_{k}+1}{\E}_{-2}({\spi^{-}})^{\lambda_{k}-\lambda_{k-1}-1}{\E}_{-3}... {\F}_{\lambda_{i-1}-\lambda_{i+1}+4}{\E}_{-k+i-1}({\spi^{-}})^{\lambda_{i+1}-\lambda_{i}-1}{\E}_{-k+i-2}({\spi^{-}})^{\lambda_{i}-\lambda_{i-1}-1}\bo_{(k-i,N-k+i)} \\
	& \cong ... \\
	& \cong ({\spi^{-}})^{-\lambda_{k}+1}{\E}_{-2}({\spi^{-}})^{\lambda_{k}-\lambda_{k-1}-1}{\F}_{\lambda_{i-1}-\lambda_{k-1}+k-i+2}{\E}_{-3}({\spi^{-}})^{\lambda_{k-1}-\lambda_{k-2}-1}... {\E}_{-k+i-2}({\spi^{-}})^{\lambda_{i}-\lambda_{i-1}-1}\bo_{(k-i,N-k+i)} \\
	& \cong ({\spi^{-}})^{-\lambda_{k}+1}{\E}_{-2}{\F}_{\lambda_{i-1}-\lambda_{k}+k-i+3}({\spi^{-}})^{\lambda_{k}-\lambda_{k-1}-1}{\E}_{-3}({\spi^{-}})^{\lambda_{k-1}-\lambda_{k-2}-1}... {\E}_{-k+i-2}({\spi^{-}})^{\lambda_{i}-\lambda_{i-1}-1}\bo_{(k-i,N-k+i)}.
\end{align*} To see the commutativity of the rightmost term  ${\E}_{-2}{\F}_{\lambda_{i-1}-\lambda_{k}+k-i+3}\bo_{(0,N)}$, since $0\leq \lambda_{i-1}-\lambda_{k} \leq N-k$ and $k \geq i \geq  2$, we have $1 \leq k-i+1 \leq \lambda_{1}-\lambda_{k}+k-i+1 \leq N-i+1 \leq N-1$. Thus ${\E}_{-2}{\F}_{\lambda_{i-1}-\lambda_{k}+k-i+3}\bo_{(0,N)} \cong {\F}_{\lambda_{1}-\lambda_{k}+k-i+3}{\E}_{-2}\bo_{(0,N)} \cong 0$ and we prove the result.

\end{proof}	
	
The first term we have to calculate is  ${\E}_{-k-1}({\spi^{-}})^{\lambda_{1}-2}{\F}_{\lambda_{1}}\bo_{(k-1,N-k+1)}$, we get 
	\begin{equation*}
		{\E}_{-k-1}({\spi^{-}})^{\lambda_{1}-2}{\F}_{\lambda_{1}}\bo_{(k-1,N-k+1)}={\E}_{-k-1}{\F}_{2}({\spi^{-}})^{\lambda_{1}-2}\bo_{(k-1,N-k+1)}[-\lambda_{1}+2]
	\end{equation*} and  we have the following exact triangle
	\begin{equation*}
		{\E}_{-k-1}{\F}_{2}\bo_{(k-1,N-k+1)} \rightarrow {\F}_{2}{\E}_{-k-1}\bo_{(k-1,N-k+1)}  \rightarrow {\spi}^{-}\bo_{(k-1,N-k+1)}.
	\end{equation*}
	
	Thus to know (\ref{eq 5}) it suffices to know the following two terms.
	\begin{equation} \label{eq 6}
		({\spi^{-}})^{-\lambda_{k}+1}{\E}_{-2}({\spi^{-}})^{\lambda_{k}-\lambda_{k-1}-1}{\E}_{-3}...({\spi^{-}})^{\lambda_{2}-\lambda_{1}-1}{\F}_{2}{\E}_{-k-1}({\spi^{-}})^{\lambda_{1}-2}{\F}_{\lambda_{2}} ... {\F}_{\lambda_{k}} \bo_{(0,N)}[\sum_{i=2}^{k}\lambda_{i}-k+2].
	\end{equation} and 
	\begin{equation} \label{eq 7}
		({\spi^{-}})^{-\lambda_{k}+1}{\E}_{-2}({\spi^{-}})^{\lambda_{k}-\lambda_{k-1}-1}{\E}_{-3}...{\E}_{-k}({\spi^{-}})^{\lambda_{2}-2}{\F}_{\lambda_{2}} ... {\F}_{\lambda_{k}} \bo_{(0,N)}[\sum_{i=2}^{k}\lambda_{i}-k+2].
	\end{equation}
	
By Lemma \ref{simplem}, (\ref{eq 6})=0 and thus 
\begin{equation} \label{eq 8}
(\ref{eq 5})=({\spi^{-}})^{-\lambda_{k}+1}{\E}_{-2}({\spi^{-}})^{\lambda_{k}-\lambda_{k-1}-1}{\E}_{-3}...{\E}_{-k}({\spi^{-}})^{\lambda_{2}-2}{\F}_{\lambda_{2}} ... {\F}_{\lambda_{k}} \bo_{(0,N)}[\sum_{i=2}^{k}\lambda_{i}-k+1].
\end{equation}

	The next step is to keep simplifying (\ref{eq 8}). Again we have \begin{equation*}
		{\E}_{-k}({\spi^{-}})^{\lambda_{2}-2}{\F}_{\lambda_{2}}\bo_{(k-2,N-k+2)} \cong {\E}_{-k}{\F}_{2}({\spi^{-}})^{\lambda_{2}-2}\bo_{(k-2,N-k+2)}[-\lambda_{2}+2],
	\end{equation*} and the following exact triangle
	\begin{equation*}
		{\E}_{-k}{\F}_{2}\bo_{(k-2,N-k+2)} \rightarrow {\F}_{2}{\E}_{-k}\bo_{(k-2,N-k+2)}  \rightarrow {\spi}^{-}\bo_{(k-2,N-k+2)}.
	\end{equation*}
	
	So to know (\ref{eq 8}), it suffices to know the following two terms 
	\begin{equation} \label{eq 9}
		({\spi^{-}})^{-\lambda_{k}+1}{\E}_{-2}({\spi^{-}})^{\lambda_{k}-\lambda_{k-1}-1}{\E}_{-3}...({\spi^{-}})^{\lambda_{3}-\lambda_{2}-1}{\F}_{2}{\E}_{-k}({\spi^{-}})^{\lambda_{2}-2}{\F}_{\lambda_{3}} ... {\F}_{\lambda_{k}} \bo_{(0,N)}[\sum_{i=3}^{k}\lambda_{i}-k+3] 
	\end{equation} and 
	\begin{equation} \label{eq 10}
		({\spi^{-}})^{-\lambda_{k}+1}{\E}_{-2}({\spi^{-}})^{\lambda_{k}-\lambda_{k-1}-1}{\E}_{-3}...{\E}_{-k+1}({\spi^{-}})^{\lambda_{3}-2}{\F}_{\lambda_{3}} ... {\F}_{\lambda_{k}} \bo_{(0,N)}[\sum_{i=3}^{k}\lambda_{i}-k+3].
	\end{equation}
	
	Using Lemma \ref{simplem} again, (\ref{eq 9})=0 and thus 
	\begin{equation} \label{eq 11}
		(\ref{eq 8}) \cong  ({\spi^{-}})^{-\lambda_{k}+1}{\E}_{-2}({\spi^{-}})^{\lambda_{k}-\lambda_{k-1}-1}{\E}_{-3}...{\E}_{-k+1}({\spi^{-}})^{\lambda_{3}-2}{\F}_{\lambda_{3}} ... {\F}_{\lambda_{k}} \bo_{(0,N)}[\sum_{i=3}^{k}\lambda_{i}-k+2].
	\end{equation} 

	Continuing this process we have the following 
	\begin{align*}
		&(\ref{eq 5}) \cong (\ref{eq 8}) \cong (\ref{eq 11}) \\
		& \cong ({\spi^{-}})^{-\lambda_{k}+1}{\E}_{-2}({\spi^{-}})^{\lambda_{k}-\lambda_{k-1}-1}{\E}_{-3}...{\E}_{-k+2}({\spi^{-}})^{\lambda_{4}-2}{\F}_{\lambda_{4}} ... {\F}_{\lambda_{k}} \bo_{(0,N)}[\sum_{i=4}^{k}\lambda_{i}-k+3] \\
		& \cong ... \\
		& \cong ({\spi^{-}})^{-\lambda_{k}+1}{\E}_{-2}({\spi^{-}})^{\lambda_{k}-2} {\F}_{\lambda_{k}} \bo_{(0,N)}[\lambda_{k}-1]  \cong ({\spi^{-}})^{-\lambda_{k}+1}{\E}_{-2}{\F}_{2} ({\spi^{-}})^{\lambda_{k}-2} \bo_{(0,N)}[1].
	\end{align*}
	
	Since that ${\E}_{-2}{\F}_{2}\bo_{(0,N)} \cong {\spi}^{-}[-1]$, we get 
	\begin{equation*}
		({\spi^{-}})^{-\lambda_{k}+1}{\E}_{-2}{\F}_{2} ({\spi^{-}})^{\lambda_{k}-2} \bo_{(0,N)}[1] \cong ({\spi^{-}})^{-\lambda_{k}+1} {\spi}^{-}[-1]({\spi^{-}})^{\lambda_{k}-2} \bo_{(0,N)}[1] \cong \bo_{(0,N)}.
	\end{equation*}
	
	The above argument shows that $({\F}_{\blam}\bo_{(0,N)})^{R}{\F}_{\blam}\bo_{(0,N)} \cong \bo_{(0,N)}$ where $\blam \in P(N-k,k)$, and thus implies the property (1).
	
	Next, we prove the second property. Given $\blam=(\lambda_{1},...,\lambda_{k}), \ \blam'=(\lambda'_{1},...,\lambda'_{k}) \in P(N-k,k)$. We assume that $\blam <_{l} \blam'$; i.e., $ \lambda_{a} \leq \lambda'_{a}$ for all $1 \leq a \leq n$ and $\lambda_{j}<\lambda'_{j}$ for some $1 \leq j \leq k$. We define 
	\begin{equation*}
		i=\text{min}\{ 1 \leq j \leq k \ | \ \lambda_{j} < \lambda'_{j}\},	
	\end{equation*} then we have $\lambda_{a}=\lambda'_{a}$ for all $1 \leq a \leq i-1$.
	
	To prove that $\Hom({\F}_{\blam}\bo_{(0,N)},{\F}_{\blam'}\bo_{(0,N)}) \cong 0$, the idea is still apply the right adjunction first to get $\Hom(\bo_{(0,N)},({\F}_{\blam}\bo_{(0,N)})^{R}{\F}_{\blam'}\bo_{(0,N)})$ and show that $({\F}_{\blam}\bo_{(0,N)})^{R}{\F}_{\blam'}\bo_{(0,N)} \cong 0$.
	
	Like the proof of property (1), we have to simplify 
	\begin{equation} \label{eq 12}
		({\spi^{-}})^{-\lambda_{k}+1}{\E}_{-2}({\spi^{-}})^{\lambda_{k}-\lambda_{k-1}-1}{\E}_{-3}({\spi^{-}})^{\lambda_{k-1}-\lambda_{k-2}-1}...{\E}_{-k-1}({\spi^{-}})^{\lambda_{1}-2}{\F}_{\lambda'_{1}}{\F}_{\lambda'_{2}} ... {\F}_{\lambda'_{k}} \bo_{(0,N)}[\sum_{i=1}^{k}\lambda_{i}-k].
	\end{equation}
	
In the rest of the proof, the homological shifts will not affect the result. So we will ignore the homological shifts for simplification. Since $\lambda_{j}=\lambda'_{j}$ for all $1 \leq j \leq i-1$, using the arguments in the proof of property (1) we obtain 
\begin{align}
	\begin{split} \label{eq 20}
	(\ref{eq 12})&\cong ({\spi^{-}})^{-\lambda_{k}+1}{\E}_{-2}({\spi^{-}})^{\lambda_{k}-\lambda_{k-1}-1}{\E}_{-3}...{\E}_{-k}({\spi^{-}})^{\lambda_{2}-2}{\F}_{\lambda'_{2}} ... {\F}_{\lambda'_{k}} \bo_{(0,N)} \\
	&\cong ... \\
	&\cong  ({\spi^{-}})^{-\lambda_{k}+1}{\E}_{-2}({\spi^{-}})^{\lambda_{k}-\lambda_{k-1}-1}{\E}_{-3}...{\E}_{-k-2+i}({\spi^{-}})^{\lambda_{i}-2}{\F}_{\lambda'_{i}} ... {\F}_{\lambda'_{k}} \bo_{(0,N)}.
	\end{split}
\end{align}
	
We need the following lemma which is similar to Lemma \ref{simplem}.

\begin{lemma}
We show that 
\begin{equation*}
({\spi^{-}})^{-\lambda_{k}+1}{\E}_{-2}({\spi^{-}})^{\lambda_{k}-\lambda_{k-1}-1}{\E}_{-3}...{\E}_{-k-2+i}({\spi^{-}})^{\lambda_{i}-2}{\F}_{\lambda'_{i}}\bo_{(k-i,N-k+i)} \cong 0
\end{equation*} for all $\blam=(\lambda_{1},...,\lambda_{k}), \ \blam'=(\lambda'_{1},...,\lambda'_{k}) \in P(N-k,k)$ with $\lambda_{i}<\lambda'_{i}$.
\end{lemma} 

\begin{proof}
First, we have to calculate 
\begin{equation*}
	{\E}_{-k-2+i}({\spi^{-}})^{\lambda_{i}-2}{\F}_{\lambda'_{i}}\bo_{(k-i,N-k+i)}\cong {\E}_{-k-2+i}{\F}_{\lambda'_{i}-\lambda_{i}+2}({\spi^{-}})^{\lambda_{i}-2}\bo_{(k-i,N-k+i)}.
\end{equation*}

Since $0 \leq \lambda_{i} < \lambda'_{i} \leq N-k$, we have $1 \leq \lambda'_{i}-\lambda_{i} \leq N-k$. Thus $-k+i+1 \leq \lambda'_{i}-\lambda_{i}-k+i \leq N-2k+i$. Because $k \geq i \geq 1$, we have $N-2k+i \leq N-k+i-1$ which implies that ${\E}_{-k-2+i}{\F}_{\lambda'_{i}-\lambda_{i}+2}\bo_{(k-i,N-k+i)} \cong {\F}_{\lambda'_{i}-\lambda_{i}+2}{\E}_{-k-2+i}\bo_{(k-i,N-k+i)}$.
Thus 
\begin{align}
\begin{split} \label{eq 13}
&({\spi^{-}})^{-\lambda_{k}+1}{\E}_{-2}({\spi^{-}})^{\lambda_{k}-\lambda_{k-1}-1}{\E}_{-3}...{\E}_{-k-2+i}({\spi^{-}})^{\lambda_{i}-2}{\F}_{\lambda'_{i}}\bo_{(k-i,N-k+i)} \\
&\cong ({\spi^{-}})^{-\lambda_{k}+1}{\E}_{-2}({\spi^{-}_{i}})^{\lambda_{k}-\lambda_{k-1}-1}{\E}_{-3}...({\spi^{-}})^{\lambda_{i+1}-\lambda_{i}-1}{\F}_{\lambda'_{i}-\lambda_{i}+2}{\E}_{-k-2+i}({\spi^{-}})^{\lambda_{i}-2}\bo_{(k-i,N-k+i)}.
\end{split}
\end{align}

The next thing we have to calculate is 
\begin{equation*}
	{\E}_{-k-1+i}({\spi^{-}})^{\lambda_{i+1}-\lambda_{i}-1}{\F}_{\lambda'_{i}-\lambda_{i}+2}\bo_{(k-i-1,N-k+i+1)} \cong {\E}_{-k-1+i}{\F}_{\lambda'_{i}-\lambda_{i+1}+3}({\spi^{-}})^{\lambda_{i+1}-\lambda_{i}-1}\bo_{(k-i-1,N-k+i+1)}.
\end{equation*}

Since $0 \leq \lambda_{i+1} \leq \lambda_{i} < \lambda'_{i} \leq N-k$, we get $1 \leq \lambda'_{i}-\lambda_{i+1} \leq N-k$. So $-k+3+i \leq \lambda'_{i}-\lambda_{i+1}-k+2+i \leq N-2k+2+i$. Now $k \geq i+1 \geq 2$, we have $N-2k+2+i \leq N-k+i$ which implies that ${\E}_{-k-1+i}{\F}_{\lambda'_{i}-\lambda_{i+1}+3}\bo_{(k-i-1,N-k+i+1)} \cong {\F}_{\lambda'_{i}-\lambda_{i+1}+3}{\E}_{-k-1+i}\bo_{(k-i-1,N-k+i+1)}$.

Continuing this process, we have 
\begin{align*}
(\ref{eq 13}) & \cong ({\spi^{-}})^{-\lambda_{k}+1}{\E}_{-2}({\spi^{-}})^{\lambda_{k}-\lambda_{k-1}-1}{\E}_{-3}...{\F}_{\lambda'_{i}-\lambda_{i+1}+3}{\E}_{-k-1+i}({\spi^{-}})^{\lambda_{i+1}-\lambda_{i}-1}{\E}_{-k-2+i}({\spi^{-}})^{\lambda_{i}-2}\bo_{(k-i,N-k+i)} \\
	& \cong ...\\
	& \cong ({\spi^{-}})^{-\lambda_{k}+1}{\E}_{-2}({\spi^{-}})^{\lambda_{k}-\lambda_{k-1}-1}{\F}_{\lambda'_{i}-\lambda_{k-1}+k-i+1}...{\E}_{-k-2+i}({\spi^{-}})^{\lambda_{i}-2}\bo_{(k-i,N-k+i)}
\end{align*}

Finally, we end up with
\begin{equation*}
{\E}_{-2}({\spi^{-}})^{\lambda_{k}-\lambda_{k-1}-1}{\F}_{\lambda'_{i}-\lambda_{k-1}+k-i+1}\bo_{(0,N)} \cong {\E}_{-2}{\F}_{\lambda'_{i}-\lambda_{k}+k-i+2}({\spi^{-}})^{\lambda_{k}-\lambda_{k-1}-1}\bo_{(0,N)}
\end{equation*} similarly argument shows that $k-i+1 \leq \lambda'_{i}-\lambda_{k}+k-i \leq N-i$. Since $i \geq 1$ and $k-i\geq0$, we get 
\begin{equation*}
	{\E}_{-2}{\F}_{\lambda'_{i}-\lambda_{k}+k-i+2}({\spi^{-}})^{\lambda_{k}-\lambda_{k-1}-1}\bo_{(0,N)} \cong {\F}_{\lambda'_{i}-\lambda_{k}+k-i+2}({\spi^{-}})^{\lambda_{k}-\lambda_{k-1}-1}{\E}_{-2}\bo_{(0,N)} \cong 0
\end{equation*} which implies that $(\ref{eq 13}) \cong 0$.
\end{proof}

As a consequence, $(\ref{eq 20}) \cong 0$ and we prove property (2). The proof is complete.
\end{proof}

Now we give the proof of Theorem \ref{Proposition 4}, where the idea is keep applying Theorem \ref{Proposition 3}.

\begin{proof}[Proof of Theorem \ref{Proposition 4}]

For the first property. Given $\blam(i)=(\lambda_{i,1},...,\lambda_{i,}) \in P(k_{i+1},\Bbbk_{i})$ for all $1 \leq i \leq n-1$. Applying adjunction and keep using Theorem \ref{Proposition 3}, we have 
\begin{align*} 
		&\Hom({\F}_{1,\blam(1)}{\F}_{2,\blam(2)} ... {\F}_{n-1, \blam(n-1)} \bo_{\eta},{\F}_{1,\blam(1)}{\F}_{2,\blam(2)} ... {\F}_{n-1, \blam(n-1)} \bo_{\eta})   \\
		&\cong 	\Hom( \bo_{\eta},({\F}_{1,\blam(1)}{\F}_{2,\blam(2)} ... {\F}_{n-1, \blam(n-1)})^{R}{\F}_{1,\blam(1)}{\F}_{2,\blam(2)} ... {\F}_{n-1, \blam(n-1)} \bo_{\eta})  \\
		& \cong \Hom( \bo_{\eta}, ({\F}_{n-1, \blam(n-1)})^{R}...({\F}_{2,\blam(2)})^{R}({\F}_{1,\blam(1)})^{R}{\F}_{1,\blam(1)}{\F}_{2,\blam(2)} ... {\F}_{n-1, \blam(n-1)} \bo_{\eta}) \\
		&\cong \Hom( \bo_{\eta}, ({\F}_{n-1, \blam(n-1)})^{R}...({\F}_{2,\blam(2)})^{R}{\F}_{2,\blam(2)} ... {\F}_{n-1, \blam(n-1)} \bo_{\eta}) \\
		& \cong ... \\
		& \cong  \Hom( \bo_{\eta}, \bo_{\eta})
\end{align*} which proves the first property.
	
	To prove the second property, given $\blam(i), \blam(i)' \in P(k_{i+1},\Bbbk_{i})$ for all $1 \leq i \leq n-1$ and assume that $(\blam(1),...,\blam(n-1)) <_{pl} (\blam(1)',...,\blam(n-1)')$. This means that there exist $1 \leq i \leq n-1$ such that $\blam(j)=\blam(j)$ for all $0 \leq j \leq i-1$ and $\blam(i) <_{l} \blam(i)'$. 
	
	 Then applying the adjunction and Theorem \ref{Proposition 3} we get 
\begin{align*}
		&\Hom({\F}_{1,\blam(1)}{\F}_{2,\blam(2)} ... {\F}_{n-1, \blam(n-1)}\bo_{\eta},{\F}_{1,\blam(1)'}{\F}_{2,\blam(2)'} ... {\F}_{n-1, \blam(n-1)'} \bo_{\eta}) \\
		& \cong 	\Hom(\bo_{\eta},({\F}_{1,\blam(1)}{\F}_{2,\blam(2)} ... {\F}_{n-1, \blam(n-1)} \bo_{\eta})^{R}{\F}_{1,\blam(1)'}{\F}_{2,\blam(2)'} ... {\F}_{n-1, \blam(n-1)'} \bo_{\eta}) \\
		& \cong \Hom(\bo_{\eta},({\F}_{n-1, \blam(n-1)})^{R}...({\F}_{2,\blam(2)})^{R}({\F}_{1,\blam(1)})^{R}{\F}_{1,\blam(1)'}{\F}_{2,\blam(2)'} ... {\F}_{n-1, \blam(n-1)'} \bo_{\eta}) \\
		&\cong  \Hom(\bo_{\eta},({\F}_{n-1, \blam(n-1)})^{R}...({\F}_{3,\blam(3)})^{R}({\F}_{2,\blam(2)})^{R}{\F}_{2,\blam(2)'} ... {\F}_{n-1, \blam(n-1)'} \bo_{\eta}) \\
		& \cong ... \\
		& \cong  \Hom(\bo_{\eta},({\F}_{n-1, \blam(n-1)})^{R}...({\F}_{i+1,\blam(i+1)})^{R}({\F}_{i,\blam(i)})^{R}{\F}_{i,\blam(i)'} ... {\F}_{n-1, \blam(n-1)'}
		\bo_{\eta}) \\
		& \cong \Hom(\bo_{\eta},0) \cong 0
	\end{align*} where the last isomorphism using the fact that $\blam(i) <_{l} \blam(i)'$ and $({\F}_{i,\blam(i)})^{R}{\F}_{i,\blam(i)'} \cong 0$ in the proof of Theorem \ref{Proposition 3}.
	
	The proof is complete.
\end{proof} 

From Theorem \ref{Theorem 1}, we have the dual exceptional collection
$\langle \ \s_{\bmu}\CC^N/\V[-|\bmu|] \ \rangle_{\bmu \in P(k,N-k)}$. We can omit the homological shifts in each term since it is used in order to make $\Hom$ concentrate in homological degree 0 in the definition of dual exceptional collection. Thus we still obtain an exceptional collection $\langle \ \s_{\bmu}\CC^N/\V \ \rangle_{\bmu \in P(k,N-k)}$.

Now we consider the dual (as vector space) of the exceptional collection $\langle \ \s_{\bmu}\CC^N/\V \ \rangle_{\bmu \in P(k,N-k)}$. It is easy to see that $\langle \ \s_{\bmu}(\CC^N/\V)^{\vee} \ \rangle_{\bmu \in P(k,N-k)}$ is again an exceptional collection with the opposite order; i.e., $\Hom( \s_{\bmu}(\CC^N/\V)^{\vee}, \s_{\bmu'}(\CC^N/\V)^{\vee})=0$ for $\bmu <_{l} \bmu'$ where $\bmu, \bmu' \in P(k,N-k)$.

We define the functor 
\begin{equation}
	{\E}_{-\bmu}\bo_{(N,0)}:={\E}_{-\mu_{1}}...{\E}_{-\mu_{N-k}}\bo_{(N,0)} \in \Hom(\Kk(N,0),\Kk(k,N-k))
\end{equation} where $\bmu=(\mu_{1},...,\mu_{N-k}) \in P(k,N-k)$. 

Then like Theorem \ref{Proposition 3} we have the following  result for ${\E}_{-\bmu}\bo_{(N,0)}$ whose proof is exactly the same by using relations (4)(a), (5)(a), and (6)(a)(c) in Definition \ref{definition 2}. 

\begin{theorem} \label{Theorem 6}
	Given a partial categorical $\dot{\Uu}_{0,N}(L\SL_2)$ action $\Kk$. Considering the functors 
	\begin{equation*}
		{\E}_{-\bmu}\bo_{(N,0)}:={\E}_{-\mu_{1}}...{\E}_{-\mu_{N-k}}\bo_{(N,0)} \in \Hom(\Kk(N,0),\Kk(k,N-k))
	\end{equation*} where $\bmu=(\mu_{1},...,\mu_{N-k}) \in P(k,N-k)$. Then $\{{\E}_{-\bmu}\bo_{(N,0)}\}_{\bmu \in P(k,N-k)}$ satisfy the following properties
	\begin{enumerate}
		\item  $\Hom({\E}_{-\bmu}\bo_{(N,0)},{\E}_{-\bmu}\bo_{(N,0)}) \cong \Hom(\bo_{(N,0)},\bo_{(N,0)})$,
		\item  $\Hom({\E}_{-\bmu}\bo_{(N,0)},{\E}_{-\bmu'}\bo_{(N,0)}) \cong 0$ if $\bmu <_{l} \bmu'$.
	\end{enumerate}
\end{theorem}

\subsection{Consequences}

In this section, we mention some corollaries of the main theorem. All of them easily follow from Theorem \ref{Proposition 4}, \ref{Proposition 3}, \ref{Theorem 6}.

The following one is easy to see from property (1) in Theorem \ref{Proposition 3}.
\begin{corollary} \label{Corollary 1}
	Given a partial categorical $\dot{\Uu}_{0,N}(L\SL_2)$ action $\Kk$. All the functors ${\F}_{\blam}\bo_{(0,N)} \in \Hom(\Kk(0,N),\Kk(k,N-k))$ where $\blam \in P(N-k,k)$  are fully faithful.
\end{corollary}

Next, the subcategories generated by the essential image of the collection of such functors gives rise to a $\SOD$ of the weight category $\Kk(k,N-k)$.
\begin{corollary} \label{Theorem 4}
	Given a partial categorical $\dot{\Uu}_{0,N}(L\SL_2)$ action $\Kk$. We denote 
	$\textnormal{Im}{\F}_{\blam}\bo_{(0,N)}$ to be the minimal full triangulated subcategories of $\Kk(k,N-k)$ generated by the class of objects which are the essential images of ${\F}_{\blam}\bo_{(0,N)}$. Then we have the following $\SOD$ 
	\begin{equation*}
		\Kk(k,N-k)=\langle \Aa(k,N-k), \ \textnormal{Im}{\F}_{\blam}\bo_{(0,N)}  \rangle_{\blam \in P(N-k,k)}
	\end{equation*} where $\Aa(k,N-k):=\langle \textnormal{Im}{\F}_{\blam}\bo_{(0,N)} \rangle_{\blam \in P(N-k,k)}^{\perp}$ is the orthogonal complement.
\end{corollary}

\begin{proof}
The main idea is to prove that $\{  \textnormal{Im}{\F}_{\blam}\bo_{(0,N)} \}_{\blam \in P(N-k,k)}$ forms a semiorthogonal sequence of admissible subcategories in $\Kk(k,N-k)$. Then using Lemma \ref{lemma 2} and by the definition of $\Aa(k,N-k)$ we have 
\begin{equation*}
	\langle \Aa(k,N-k), \textnormal{Im}{\F}_{\blam}\bo_{(0,N)}  \rangle_{\blam \in P(N-k,k)}
\end{equation*} is a $\SOD$ for $\Kk(k,N-k)$.

The semiorthogonal property for $\{  \textnormal{Im}{\F}_{\blam}\bo_{(0,N)} \}_{\blam \in P(N-k,k)}$ is easily followed from property (2) of Theorem \ref{Proposition 3}. Next, since from the definition of categorical action the left and right adjoint functors $({\F}_{\blam}\bo_{(0,N)})^{L}$, $({\F}_{\blam}\bo_{(0,N)})^{R}$ both exist, we have $\textnormal{Im}{\F}_{\blam}\bo_{(0,N)}$ are admissible subcategories.

The proof is complete.
\end{proof}

We have similar results for the functors ${\E}_{-\bmu}\bo_{(N,0)}$, which are summarized as the following corollary where the proof is omitted since it is pretty much the same as Corollary \ref{Theorem 4}.

\begin{corollary} \label{corollary}
We have $\{{\E}_{-\bmu}\bo_{(N,0)}\}_{\bmu \in P(k,N-k)}$ are fully faithful functors. We define 
$\textnormal{Im}{\E}_{-\bmu}\bo_{(N,0)}$ to be the minimal full triangulated subcategories of $\Kk(k,N-k)$ generated by the class of objects which are the essential images of ${\E}_{-\bmu}\bo_{(N,0)}$. Then we have the following $\SOD$ 
\begin{equation*}
	\Kk(k,N-k)=\langle \Bb(k,N-k), \textnormal{Im}{\E}_{-\bmu}\bo_{(N,0)} \rangle_{\bmu \in P(k,N-k)}
\end{equation*} where $\Bb(k,N-k):=\langle \textnormal{Im}{\E}_{-\bmu}\bo_{(N,0)} \rangle_{\bmu \in P(k,N-k)}^{\perp}$ is the orthogonal complement.
\end{corollary}

Now we extend the above results from $\SL_{2}$ to the general $\SL_{n}$ version, and again the proof is omitted due to the similarity.

\begin{corollary} \label{Corollary 2}
\sloppy Given a partial categorical $\dot{\Uu}_{0,N}(L\SL_n)$ action $\Kk$. The functors ${\F}_{1,\blam(1)}{\F}_{2,\blam(2)} ... {\F}_{n-1, \blam(n-1)} \bo_{\eta} \in \Hom(\Kk(\eta),\Kk(\kk))$, where $\blam(i)=(\blam(i)_{1},...,\blam(i)_{\Bbbk_{i}}) \in P(k_{i+1},\Bbbk_{i})$ for all $1 \leq i \leq n-1$, are all fully faithful.
\end{corollary}

\begin{corollary} \label{Theorem 5}
\sloppy Given a partial categorical $\dot{\Uu}_{0,N}(L\SL_n)$ action $\Kk$.  We denote  
$\textnormal{Im}{\F}_{1,\blam(1)}{\F}_{2,\blam(2)} ... {\F}_{n-1, \blam(n-1)} \bo_{\eta}$ to be the minimal full triangulated subcategories of $\Kk(\kk)$ generated by the class of objects which are the essential images of ${\F}_{1,\blam(1)}{\F}_{2,\blam(2)} ... {\F}_{n-1, \blam(n-1)} \bo_{\eta}$ where $\blam(i)=(\blam(i)_{1},...,\blam(i)_{\Bbbk_{i}}) \in P(k_{i+1},\Bbbk_{i})$ for all $1 \leq i \leq n-1$. Then we have the following $\SOD$
\begin{equation*}
	\Kk(\kk)=\langle \Aa(\kk), \textnormal{Im}{\F}_{1,\blam(1)}{\F}_{2,\blam(2)} ... {\F}_{n-1, \blam(n-1)} \bo_{\eta} \rangle_{\blam(i) \in P(k_{i+1},\Bbbk_{i})}
\end{equation*} where $\Aa(\kk)=\langle \textnormal{Im}{\F}_{1,\blam(1)}{\F}_{2,\blam(2)} ... {\F}_{n-1, \blam(n-1)} \bo_{\eta}  \rangle_{\blam(i) \in P(k_{i+1},\Bbbk_{i})}^{\perp}$ is the orthogonal complement.
\end{corollary}


\section{Application to Grassmannian of coherent sheaves}

Since the Kapranov exceptional collection is well-known, we would like to see examples (besides the usual Grassmanninas and partial flag varieties) where we can apply the above results to get $\SOD$s. Also, we believe that our result about the categorical action of $\dot{\Uu}_{0,N}(L\SL_{2})$ on $\bigoplus_{k}\Dd^b(\Gr(k,\CC^N))$ can be easily extended to a categorical action of $\dot{\Uu}_{0,N}(L\SL_{2})$ on Grassmannian bundles, i.e. $\bigoplus_{k}\Dd^b(\Gr(k,\Ee))$ where $\Ee$ is a locally free sheaf of rank $N$ on a smooth projective variety.

Thus the next possible example we consider is a generalization to Grassmannians (more precisely, relative Quot schemes) of coherent sheaves of homological dimension $\leq 1$.  Let $X$ be a connected smooth projective variety and $\GGG$ be a coherent sheaf on $X$ of homological dimension $\leq 1$, i.e. such that there is a locally free resolution $0 \rightarrow \Ee^{-1} \rightarrow \Ee^{0} \rightarrow \GGG \rightarrow 0$. In order to reduce the notations, we still denote the rank of $\GGG$ to be $N \coloneqq \text{rk}\Ee^{0}-\text{rk}\Ee^{-1} \geq 2$. Then we consider the Grassmannians/relative Quot scheme $\Gr(\GGG, k)$ of rank $k$ locally free quotients of $\GGG$. By Proposition A.1 in \cite{AT}, $\Gr(\GGG,k)$ is smooth, thus we can consider its bounded derived category of coherent sheaves $\Dd^b(\Gr(\GGG,k))$.

In this section, we apply the main results (in particular Theorem \ref{Proposition 3} and \ref{Theorem 6}) to obtain a $\SOD$ on $\Dd^b(\Gr(\GGG,k))$. We prove it by defining certain functors ${\E}_{r}\bo_{(N-k,k)}$, ${\F}_{s}\bo_{(N-k,k)}$ via using a correspondence like diagram (\ref{diag 1}) and show that they satisfy the conditions in Definition \ref{definition 2}.

Consider the following correspondence 
\begin{equation} \label{diag 2}
	\xymatrix{ 
		&&\Fl(\GGG,k+1,k)=\{\GGG \overset{N-k-1}{\twoheadrightarrow} \W' \overset{1}{\twoheadrightarrow} \W \overset{k}{\rightarrow} 0  \} 
		\ar[ld]_{p_1} \ar[rd]^{p_2}   \\
		& \Gr(\GGG,k)  && \Gr(\GGG,k+1)
	}
\end{equation}  where $\Fl(\GGG,k+1,k)$ is the 3-step partial flag variety which parametrizing successive locally free quotients of rank $k+1$, $k$ of the coherent sheaf $\GGG$. The numbers above the arrow indicate the decreasing of ranks, and $p_1$, $p_2$ are the natural projections.

Let $\pi:\Gr(\GGG,k) \rightarrow X$ be the natural projection. Then $\Gr(\GGG,k)$ carries the following tautological short exact sequence
\begin{equation} \label{ses1} 
	0 \rightarrow \Ss_{\pi} \rightarrow \pi^*\GGG \rightarrow \Qq_{\pi} \rightarrow 0
\end{equation} where $\Qq_{\pi}$ is the universal quotient bundle whereas the universal subsheaf $\Ss_{\pi}$ may not be locally free. Similarly for $\pi':\Gr(\GGG,k+1) \rightarrow X$. Pulling back the universal quotient bundles $\Qq_{\pi}$, $\Qq_{\pi'}$ in the correspondence (\ref{diag 2}), we obtain the surjective morphism $\rho: p^{*}_{2}\Qq_{\pi'} \twoheadrightarrow p^{*}_{1}\Qq_{\pi}$ of tautological quotient bundles on $\Fl(\GGG,k+1,k)$. Taking the kernel we get the line bundle, denoted by $\ker(\rho)$, on $\Fl(\GGG,k+1,k)$. Then we define the following functors
\begin{equation*}
	 {\E}_{r}\bo_{(N-k,k)} \coloneqq p_{2*}(p^{*}_1 \otimes (\ker(\rho))^{r} ):\Dd^b(\Gr(\GGG,k)) \rightarrow \Dd^b(\Gr(\GGG,k+1))
\end{equation*} with $-N+k \leq r \leq 0$ and similarly for ${\F}_{s}\bo_{(N-k,k)}$ with $0 \leq s \leq k$ in the opposite direction. 

To check that the functors defined above satisfy the conditions in Definition \ref{definition 2}, it remains to define the functors ${\spi}^{\pm}\bo_{(N-k,k)}$ on $\Dd^b(\Gr(\GGG,k))$.  There are two natural determinant line bundle on $\Gr(\GGG,k)$, one is $\det(\Qq_{\pi})$, the other is $\det(\Ss_{\pi})$. Here we use the tautological short exact sequence (\ref{ses1}) and $\GGG$ admits a two-term locally free resolution to get 
\begin{equation*}
	\det(\Ss_{\pi}) \cong \pi^*\det(\GGG) \otimes \det(\Qq_{\pi})^{-1} \cong \pi^*(\det(\Ee^{0}) \otimes\det(\Ee^{-1})^{-1}) \otimes  \det(\Qq_{\pi})^{-1}.
\end{equation*} Then we define 
\begin{align*}
	&{\spi}^{+}\bo_{(N-k,k)} \coloneqq \otimes \det(\Qq_{\pi})[1-k]  :\Dd^b(\Gr(\GGG,k)) \rightarrow \Dd^b(\Gr(\GGG,k)), \\
	&{\spi}^{-}\bo_{(N-k,k)} \coloneqq \otimes \det(\Ss_{\pi})^{-1}[1+k-N]  :\Dd^b(\Gr(\GGG,k)) \rightarrow \Dd^b(\Gr(\GGG,k)).
\end{align*} It is easy to see that both the functors ${\spi}^{+}$, ${\spi}^{-}$ are invertible.
 
Then we state the main result of this section. 

\begin{theorem} \label{Theorem 7}
The functors  ${\E}_{r}\bo_{(N-k,k)}$, ${\F}_{s}\bo_{(N-k,k)}$, ${\spi}^{\pm}\bo_{(N-k,k)}$  defined above gives a partial categorical $\dot{\Uu}_{0,N}(L\SL_{2})$ action. In particular, this implies that $\Dd^b(\Gr(\GGG,k))$ admits the following two $\SOD$s
\begin{align}
	\Dd^b(\Gr(\GGG,k)) &= \langle \Aa(N-k,k), \ \textnormal{Im}{\F}_{\blam}\bo_{(0,N)}  \rangle_{\blam \in P(k,N-k)}\label{sodf} \\
	&=\langle \Bb(N-k,k), \ \textnormal{Im}{\E}_{-\bmu}\bo_{(N,0)} \rangle_{\bmu \in P(N-k,k)}\label{sode} 
\end{align} where $\Aa(N-k,k):=\langle \textnormal{Im}{\F}_{\blam}\bo_{(0,N)} \rangle_{\blam \in P(k,N-k)}^{\perp}$, $\Bb(N-k,k):=\langle \textnormal{Im}{\E}_{-\bmu}\bo_{(N,0)} \rangle_{\bmu \in P(N-k,k)}^{\perp}$ are the orthogonal complements.
\end{theorem}

To prove this theorem, note that all the functors ${\E}_{r}\bo_{(N-k,k)}$, ${\spi}^{\pm}\bo_{(N-k,k)}$ are FM transformations and the corresponding FM kernels are 
\begin{align*}
	{\Ee}_{r}\bo_{(N-k,k)} &\coloneqq \iota_{*}\ker(\rho)^r \in \Dd^b(\Gr(\GGG,k) \times \Gr(\GGG,k+1)), \\
	\bo_{(N-k,k)}{\Ff}_{s} &\coloneqq ^{\TTt}\iota_{*}\ker(\rho)^r \in \Dd^b(\Gr(\GGG,k+1) \times \Gr(\GGG,k)), \\
	{\Psi}^{+}\bo_{(N-k,k)}&\coloneqq \Delta_{*}\det(\Qq_{\pi})[1-k] \in \Dd^b(\Gr(\GGG,k) \times \Gr(\GGG,k)), \\
	{\Psi}^{-}\bo_{(N-k,k)}&\coloneqq \Delta_{*}\det(\Ss_{\pi})^{-1}[1+k-N] \in \Dd^b(\Gr(\GGG,k) \times \Gr(\GGG,k)),  
\end{align*} respectively where $\iota:\Fl(\GGG,k+1,k) \rightarrow \Gr(\GGG,k) \times \Gr(\GGG,k+1)$ is the natural inclusion, $^{\TTt}\iota:\Fl(\GGG,k+1,k) \rightarrow \Gr(\GGG,k+1) \times \Gr(\GGG,k)$ is the transpose inclusion, and $\Delta:\Gr(\GGG,k) \rightarrow \Gr(\GGG,k) \times \Gr(\GGG,k)$ is the diagonal map. Thus we prove Theorem \ref{Theorem 7} by checking the conditions in Definition \ref{definition 2} at the level of FM kernels.

The first is condition (4).

\begin{lemma}(condition (4)) \label{lemma 3}
The right adjoints of ${\Ee}_{r}\bo_{(N-k,k)}$ and ${\Ff}_{s}\bo_{(N-k,k)}$ are given by the following 
\begin{align*}
	({\Ee}_{r}\bo_{(N-k,k)})_{R} &\cong \bo_{(N-k,k)}{({\Psi}^{+})^{r+1}} \ast {\Ff}_{k+2} \ast ({\Psi}^{+})^{-r-2}[-r-1], 
	\\	({\Ff}_{s}\bo_{(N-k,k)})_{R} &\cong \bo_{(N-k,k)}{({\Psi}^{-})^{-s+1}} \ast {\Ee}_{-N+k-2} \ast ({\Psi}^{-})^{s-2}[s-1].
\end{align*} 
\end{lemma}

\begin{proof}
We prove the case for $\Ee$, the other is similar.  From Proposition \ref{Proposition 1}, we know that 
\begin{equation} \label{rad}
	({\Ee}_{r}\bo_{(N-k,k)})_{R} = \{\iota_{*}\ker(\rho)^r \}^{\vee} \otimes \pi^{*}_{1}\omega_{\Gr(\GGG,k)} [\dim \Gr(\GGG,k)]  
\end{equation} and a standard calculation shows that 
\begin{equation*}
 \{\iota_{*}\ker(\rho)^r \}^{\vee}	\cong \iota_{*}(\ker(\rho)^{-r} \otimes \omega_{\Fl(\GGG,k+1,k)}) \otimes \omega^{-1}_{\Gr(\GGG,k) \times \Gr(\GGG,k+1)}[\dim \Fl(\GGG,k+1,k)-\dim \Gr(\GGG,k) \times \Gr(\GGG,k+1)].
\end{equation*}

An easy calculation shows that $\dim \Fl(\GGG,k+1,k)=\dim X + (N-k)k+N-k-1$ and $\dim \Gr(\GGG,k)=\dim X + (N-k)k$. Thus (\ref{rad}) becomes 
\begin{equation} \label{rad2}
\iota_{*}(\ker(\rho)^{-r} \otimes \omega_{\Fl(\GGG,k+1,k)}) \otimes \pi^{*}_{2}\omega^{-1}_{\Gr(\GGG,k+1)}  [k].
\end{equation} Using the projection $p_2:\Fl(\GGG,k+1,k) \rightarrow \Gr(\GGG,k+1)$, we have $ \omega_{\Fl(\GGG,k+1,k)} \cong  \omega_{p_2} \otimes p^{*}_{2}\omega_{\Gr(\GGG,k+1)}$, where $\omega_{p_2}$ is the relative canonical bundle. Since $\pi_{2} \circ \iota = p_2$, (\ref{rad2}) becomes  $\iota_{*}(\ker(\rho)^{-r} \otimes \omega_{p_2})[k]$.

To calculate $\omega_{rel}$, from the following short exact sequence
\begin{equation*}
	0 \rightarrow \ker(\rho) \rightarrow p^{*}_{2}\Qq_{\pi'} \overset{\rho}{\twoheadrightarrow} p^{*}_{1}\Qq_{\pi} \rightarrow 0
\end{equation*} we know that the relative cotangent bundle is given by  $\Omega_{p_2} \cong \ker(\rho) \otimes  p^{*}_{1}\Qq_{\pi}^{\vee} $. Thus $\omega_{p_2} = \bigwedge^{top} \Omega_{p_2} \cong \ker(\rho)^{k} \otimes p^{*}_1\det(\Qq_{\pi})^{-1}$. Since $\ker(\rho) \cong p^*_{2}\det(\Qq_{\pi'}) \otimes p^*_{1}\det(\Qq_{\pi})^{-1}$, we obtain
\begin{align*}
	\iota_{*}(\ker(\rho)^{-r} \otimes \omega_{p_2})[k] &\cong \iota_{*}(\ker(\rho)^{k-r} \otimes p^{*}_1\det(\Qq_{\pi})^{-1})[k] \\
	& \cong \iota_{*}(\ker(\rho)^{k+2} \otimes  p^*_{2}\det(\Qq_{\pi'})^{-r-2} \otimes p^{*}_1\det(\Qq_{\pi})^{r+1})[k] \\
	& \cong  \iota_{*}(\ker(\rho)^{k+2}) \otimes  \pi^*_{2}\det(\Qq_{\pi'})^{-r-2}[(-r-2)(-k)] \otimes \pi^{*}_{1}\det(\Qq_{\pi})^{r+1}[(r+1)(1-k)][-r-1]
\end{align*} which is isomorphic to $\bo_{(N-k,k)}{({\Psi}^{+})^{r+1}} \ast {\Ff}_{k+2} \ast ({\Psi}^{+})^{-r-2}[-r-1]$ from the definition of the kernels 
${\Psi}^{+}\bo_{(N-k,k)}$.
\end{proof}

The next is condition (5) where the proof is simply used base-change, so we leave it to the readers.

\begin{lemma}(condition (5)) \label{lemma 4}
We have the following isomorphisms of FM kernels.
\begin{align*}
	(\Psi^{\pm} \ast \Ee_{r})\bo_{(N-k,k)} &\cong (\Ee_{r+1} \ast \Psi^{\pm})\bo_{(N-k,k)} [\mp 1] ,\\
	(\Psi^{\pm} \ast \Ff_{s})\bo_{(N-k,k)} &\cong (\Ff_{s-1} \ast \Psi^{\pm})\bo_{(N-k,k)} [\pm 1].
\end{align*}
\end{lemma}

Finally, we prove the most important condition, which is condition (6).

\begin{theorem} \label{Theorem 8} 
We have the following two exact triangles in $\Dd^b(\Gr(\GGG,k) \times \Gr(\GGG,k))$
\begin{align}
	& (\Ff_{k} \ast \Ee_{0})\bo_{(N-k,k)} \rightarrow (\Ee_{0} \ast \Ff_{k})\bo_{(N-k,k)} \rightarrow \Psi^{+}\bo_{(N-k,k)}, \label{eta} \\
	& (\Ee_{-N+k} \ast \Ff_{0})\bo_{(N-k,k)} \rightarrow (\Ff_{0} \ast \Ee_{-N+k})\bo_{(N-k,k)} \rightarrow \Psi^{-}\bo_{(N-k,k)}, \label{etb}
\end{align} and $(\Ee_{r} \ast \Ff_{s})\bo_{(N-k,k)} \cong (\Ff_{s} \ast \Ee_{r})\bo_{(N-k,k)}$ if $1+k-N \leq r+s \leq k-1$.
\end{theorem}

To compare $(\Ee_{r} \ast \Ff_{s})\bo_{(N-k,k)}$ and $(\Ff_{s} \ast \Ee_{r})\bo_{(N-k,k)}$, similar to the proof of Proposition 5.11 in \cite{Hsu}, we divide the proof to two steps. The first step is to handle the case where $r=s=0$.

\begin{lemma}
$(\Ee_{0} \ast \Ff_{0})\bo_{(N-k,k)} \cong (\Ff_{0} \ast \Ee_{0})\bo_{(N-k,k)}$.
\end{lemma}

\begin{proof}
By definition, 
\begin{equation} \label{e0f0}
(\Ff_{0} \ast \Ee_{0})\bo_{(N-k,k)} \cong
\pi_{13*}(\pi^*_{12} \iota_{*}\Oo_{\Fl(\GGG,k+1,k)} \otimes \pi^{*}_{23} {^{\TTt}}\iota_{*}\Oo_{\Fl(\GGG,k+1,k)} )
\end{equation} where $\iota:\Fl(\GGG,k+1,k) \rightarrow \Gr(\GGG,k) \times \Gr(\GGG,k+1)$ is the natural inclusion and ${^{\TTt}}\iota:\Fl(\GGG,k+1,k) \rightarrow \Gr(\GGG,k+1) \times \Gr(\GGG,k)$ is the transpose inclusion. 

Using the following fibred product diagrams 
\begin{equation*} 
	\begin{tikzcd} 
		\Fl(\GGG,k+1,k) \times \Gr(\GGG,k)   \arrow[r, "\iota \times id"]  \arrow[d, "a_{1}"] & \Gr(\GGG,k) \times  \Gr(\GGG,k+1) \times  \Gr(\GGG,k) \arrow[d, "\pi_{12}"]\\
		\Fl(\GGG,k,k+1)  \arrow[r, "\iota"] &  \Gr(\GGG,k) \times  \Gr(\GGG,k+1)
	\end{tikzcd}
\end{equation*}
\begin{equation*} 
	\begin{tikzcd} 
		\Gr(\GGG,k) \times \Fl(\GGG,k+1,k)   \arrow[r, "id \times ^{\TTt}\iota"]  \arrow[d, "a_{2}"] &\Gr(\GGG,k)  \times \Gr(\GGG,k+1) \times \Gr(\GGG,k)  \arrow[d, "\pi_{23}"]\\
		\Fl(\GGG,k+1,k) \arrow[r, "^{\TTt}\iota"] & \Gr(\GGG,k+1) \times \Gr(\GGG,k) 
	\end{tikzcd}
\end{equation*} where $a_1$, $a_2$ are the natural projections, then we obtain
\begin{equation}
 (\ref{e0f0})	\cong \pi_{13*}((\iota \times id)_{*}\Oo_{\Fl(\GGG,k+1,k) \times \Gr(\GGG,k)} \otimes (id\times ^{\TTt}\iota)_{*}\Oo_{\Gr(\GGG,k) \times \Fl(\GGG,k+1,k)}). \label{e0f01}
\end{equation}

Next, the following fibred product diagram 
\begin{equation*} 
	\begin{tikzcd} 
		\Zz   \arrow[r, "b_{1}"]  \arrow[d, "b_{2}"] &  \Fl(\GGG,k+1,k)  \times \Gr(\GGG,k)  \arrow[d, "\iota \times id"]\\
		\Gr(\GGG,k)  \times  \Fl(\GGG,k+1,k)   \arrow[r, "id \times ^{\TTt}\iota"] & \Gr(\GGG,k)  \times \Gr(\GGG,k+1) \times \Gr(\GGG,k) 
	\end{tikzcd}
\end{equation*} where $\Zz$ is the subvariety of $\Gr(\GGG,k)  \times \Gr(\GGG,k+1) \times \Gr(\GGG,k)$ defined as follows
\begin{equation*}
	\Zz \coloneqq \{ (\GGG \overset{N-k}{\twoheadrightarrow} \W, \ \GGG \overset{N-k-1}{\twoheadrightarrow} \W', \ \GGG \overset{N-k}{\twoheadrightarrow} \W'') \ | \  \W' \overset{1}{\twoheadrightarrow} \W, \  \W' \overset{1}{\twoheadrightarrow} \W''  \}, 
\end{equation*} and it can also be viewed as the subvariety that parametrizing the following diagram of locally free  quotients 
\[
\xymatrix@R=0.1pc{ 
	&&&  \W \\
	& \GGG \ar@{->>}[r]^{N-k-1} & \W'  \ar@{->>}[ru]^{1} \ar@{->>}[rd]_{1} \\
	&&& \W''
}.
\]

Thus this gives us that
\begin{equation} \label{e0f02}
	 (\ref{e0f01}) \cong  \pi_{13*}(id\times ^{\TTt}\iota)_{*}b_{2*}(\Oo_{\Zz}).
\end{equation}

Finally, we have the following commutative diagram
\begin{equation*} 
	\begin{tikzcd} 
		\Zz   \arrow[r, "j_{1}"]  \arrow[d, "\pi_{13}|_{\Zz}"] & \Gr(\GGG,k)  \times \Gr(\GGG,k+1)  \times  \Gr(\GGG,k) \arrow[d, "\pi_{13}"]\\
		\Yy  \arrow[r, "t"] & \Gr(\GGG,k) \times  \Gr(\GGG,k) 
	\end{tikzcd}
\end{equation*} where $\Yy$ is the following subvariety of $\Gr(\GGG,k) \times  \Gr(\GGG,k)$
\begin{equation*}
	\Yy=\pi_{13}(\Zz)=\{(\GGG \overset{N-k}{\twoheadrightarrow}\W,\GGG \overset{N-k}{\twoheadrightarrow} \W'')\ |\ \rk(\W\cup \W'') \leq k+1 \}
\end{equation*} and $j_1=(id \times ^{\TTt}\iota) \circ b_2$, $t:\Yy \rightarrow \Gr(\GGG,k) \times \Gr(\GGG,k)$ are the inclusions.  Note that $(\pi_{13}|_{\Zz*})(\Oo_{\Zz}) \cong \Oo_{\Yy}$ and thus
\begin{equation*}
	(\Ff_{0} \ast \Ee_{0})\bo_{(N-k,k)} \cong (\ref{e0f02}) \cong t_{*}(\pi_{13}|_{\Zz*})(\Oo_{\Zz}) \cong t_{*}\Oo_{\Yy}.
\end{equation*} 

Similarly, we end up with the following diagram when calculating $(\Ee_{0} \ast \Ff_{0})\bo_{(N-k,k)}$ 
\begin{equation*} 
	\begin{tikzcd} 
		\Zz'   \arrow[r, "j_{2}"]  \arrow[d, "\pi_{13'}|_{Z'}"] &\Gr(\GGG,k) \times \Gr(\GGG,k-1) \times \Gr(\GGG,k) \arrow[d, "\pi_{13'}"]\\
		\Yy  \arrow[r, "t"] & \Gr(\GGG,k) \times  \Gr(\GGG,k)
	\end{tikzcd}
\end{equation*}  where 
\begin{equation*}
	\Zz' \coloneqq \{ (\GGG \overset{N-k}{\twoheadrightarrow} \W, \ \GGG \overset{N-k+1}{\twoheadrightarrow} \W''', \ \GGG \overset{N-k}{\twoheadrightarrow} \W'') \ | \  \W \overset{1}{\twoheadrightarrow} \W''', \  \W'' \overset{1}{\twoheadrightarrow} \W'''  \}, 
\end{equation*} or the subvariety that parametrizing the following diagram of locally free quotients 
\[
\xymatrix@R=0.1pc{ 
& & \W \ar@{->>}[rd]^{1} \\
& \GGG  \ar@{->>}[ru]^{N-k} \ar@{->>}[rd]_{N-k} & & \W'''  \\
& & \W'' \ar@{->>}[ru]_{1}
}
\]

By the same argument
\begin{equation*}
	(\Ee_{0} \ast \Ff_{0})\bo_{(N-k,k)} \cong	\pi_{13'*}j_{2*}(\Oo_{\Zz'}) \cong t_{*}(\pi_{13'}|_{\Zz'*})(\Oo_{\Zz'}) \cong t_{*}\Oo_{\Yy}
\end{equation*} which proves the lemma.

\end{proof}

Before we move to the proof of Theorem \ref{Theorem 7}, we need to introduce more tools. From the two maps $\pi_{13'}|_{Z'}:\Zz' \rightarrow \Yy$ and $\pi_{13}|_{Z}:\Zz \rightarrow \Yy$ we can form their fibred product
\begin{equation} \label{diagram1}
	\begin{tikzcd} [column sep=large]
		\Xx=\Zz \times_{\Yy} \Zz' \arrow[r, "g_{1}"]  \arrow[d, "g_{2}"] & \Zz\arrow[d, "\pi_{13}|_{\Zz}"]\\
		\Zz' \arrow[r, "\pi_{13'}|_{\Zz'}"] & \Yy
	\end{tikzcd}
\end{equation} where $g_1$ and $g_2$ are the natural projections. We denote $p:\Xx \rightarrow \Yy$ to be the natural projection. Note that $\Xx$ is given by 
\begin{equation*}
	\Xx =\{ (\GGG \overset{N-k-1}{\twoheadrightarrow} \W', \ \GGG \overset{N-k}{\twoheadrightarrow} \W, \ \GGG \overset{N-k}{\twoheadrightarrow} \W'',  \ \GGG \overset{N-k+1}{\twoheadrightarrow} \W''') \ | \  \W' \overset{1}{\twoheadrightarrow} \W, \  \W' \overset{1}{\twoheadrightarrow} \W'', \ \W \overset{1}{\twoheadrightarrow} \W''', \  \W'' \overset{1}{\twoheadrightarrow} \W'''  \}
\end{equation*} or the variety parametrizing the following diagram of locally free quotients 
\[
\xymatrix@R=0.1pc{ 
	& && \W \ar@{->>}[rd]^{1} \\
	& \GGG \ar@{->>}[r]^{N-k-1} & \W'  \ar@{->>}[ru]^{1} \ar@{->>}[rd]_{1} && \W'''  \\
	& & &\W'' \ar@{->>}[ru]_{1}
}.
\]

Roughly speaking, the idea of comparing $(\Ff_{s} \ast \Ee_{r})\bo_{(N-k,k)}$ and $(\Ee_{r} \ast \Ff_{s})\bo_{(N-k,k)}$ is to pullback them to the larger space $\Xx$ and then pushforward to $\Yy$. 
Let $\Qq'$, $\Qq$, $\Qq''$, $\Qq'''$ be the four tautological quotient bundles on $\Xx$ with natural morphisms between them fit into the following commutative diagram
\begin{equation} \label{diagQ}
	\xymatrix{
		0 \ar[r] &\ker(\rho_2) \ar[r] \ar@{.>}[d]^{s}
		& \Qq' \ar@{->>}[r]^{\rho_2} \ar@{->>}[d]^{\rho_1}
		& \Qq''  \ar[r] \ar@{->>}[d]^{\rho_4} 
		& 0\\
		0 \ar[r] &\ker(\rho_3) \ar[r] 
		& \Qq \ar@{->>}[r]^{\rho_3}  
		& \Qq'''  \ar[r]
		& 0
	}
\end{equation} where $\ker(\rho_2)$ and $\ker(\rho_3)$ are the natural line bundles on $\Xx$. Since the right square commutes, it induces a natural map $s:\ker(\rho_2) \rightarrow \ker(\rho_3)$ between line bundles. The morphism $s$ gives rise to a section of the line bundle $\ker(\rho_2)^{-1}\otimes \ker(\rho_3)$ whose zero locus cut out a divisor $D$ in $\Xx$ which consists of points where $\W=\W''$. We have the following short exact sequence
\begin{equation}\label{dses}
	0 \rightarrow \ker(\rho_2) \rightarrow \ker(\rho_3) \rightarrow \Oo_{D} \otimes \ker(\rho_3) \rightarrow 0.
\end{equation} Moreover, note that the partial flag variety $\Fl(\GGG,k,k-1)$ can be identified with $\Gr(\Qq,k-1)$ which is the projective bundle $\PP_{\sub}(\Qq)$. So the restriction of the line bundle $\ker(\rho_3)$ to the divisor $D$ is the pullback of the tautological bundle $\Oo_{\PP_{\sub}(\Qq)}(-1)$, i.e. (\ref{dses}) becomes 
\begin{equation}\label{dses'}
	0 \rightarrow \ker(\rho_2) \rightarrow \ker(\rho_3) \rightarrow \Oo_{D} \otimes \Oo_{\PP_{\sub}(\Qq)}(-1) \rightarrow 0.
\end{equation} 

On the other hand, the commutative right square again there induces a morphism $s':\ker(\rho_1) \rightarrow \ker(\rho_4)$ between line bundles. The zero locus of $s'$ cuts out the same divisor $D$ and we obtain another short exact sequence which is similar to (\ref{dses}).
\begin{equation} \label{sesd}
	0 \rightarrow \ker(\rho_1) \rightarrow \ker(\rho_4) \rightarrow \Oo_{D} \otimes \ker(\rho_4)  \rightarrow 0.
\end{equation} We tenor the short exact sequence (\ref{sesd}) by $\ker(\rho_1)^{-1} \otimes \ker(\rho_4)^{-1}$ to get 
\begin{equation} \label{sesd'}
	0 \rightarrow \ker(\rho_4)^{-1} \rightarrow \ker(\rho_1)^{-1} \rightarrow \Oo_{D} \otimes \ker(\rho_1)^{-1}  \rightarrow 0.
\end{equation}

We let $\Ss', \ \Ss, \ \Ss'', \ \Ss'''$ denote the four tautological subsheaves on $\Xx$ corresponding to $\Qq',\ \Qq,\ \Qq'',\ \Qq'''$ respectively. For example, we have $0 \rightarrow \Ss' \rightarrow \GGG \rightarrow \Qq' \rightarrow 0$ on $\Xx$. Then $\ker(\rho_1)$ also fits into the following short exact sequence
\begin{equation*} 
	0 \rightarrow \Ss' \rightarrow \Ss \rightarrow \ker(\rho_1) \rightarrow 0.
\end{equation*}  Since the partial flag variety $\Fl(\GGG,k+1,k)$ can be identified with $\Gr(\Ss,1)$ which is the projective bundle $\PP_{\quo}(\Ss)$, the restriction of $\ker(\rho_1)$ to $D$ is the pullback of the tautological bundle $\Oo_{\PP_{\quo}(\Ss)}(1)$. Thus (\ref{sesd'}) becomes 
\begin{equation}
	0 \rightarrow \ker(\rho_4)^{-1} \rightarrow \ker(\rho_1)^{-1} \rightarrow \Oo_{D} \otimes \Oo_{\PP_{\quo}(\Ss)}(-1)  \rightarrow 0.
\end{equation}

Since $\Psi^{+}\bo_{(N-k,k)}$ is invertible, from Lemma \ref{lemma 4} we have 
\begin{equation*}
[\Psi^{+}\ast \Ee_{r}\ast (\Psi^{+})^{-1} ]\bo_{(N-k,k)} \cong \Ee_{r+1}\bo_{(N-k,k)}[-1]	
\end{equation*} similarly for $\Ff_{s}\bo_{(N-k,k)}$. Applying this inductively we obtain
\begin{align*}
	& (\Ee_{r}\ast \Ff_{s})\bo_{(N-k,k)} \cong  [(\Psi^{+})^r \ast \Ee_{0} \ast \Ff_{s+r} \ast (\Psi^{+})^{-r} ]\bo_{(N-k,k)}, \\
	& (\Ff_{s}\ast \Ee_{r})\bo_{(N-k,k)}
	\cong  [(\Psi^{+})^r \ast \Ff_{r+s} \ast \Ee_{0} \ast (\Psi^{+})^{-r} ]\bo_{(N-k,k)}.
\end{align*} Thus it suffices to compare $(\Ff_{r+s} \ast \Ee_{0})\bo_{(k,N-k)}$ and $ (\Ee_{0} \ast \Ff_{r+s})\bo_{(k,N-k)}$. On the other hand, we can also simplify it so that it suffices to compare $(\Ff_{0} \ast \Ee_{r+s})\bo_{(k,N-k)}$ and $ (\Ee_{r+s} \ast \Ff_{0})\bo_{(k,N-k)}$.

\begin{proof}[Proof of Theorem \ref{Theorem 8}]
Since we define our functor ${\E}_{r}\bo_{(N-k,k)}$,  ${\F}_{s}\bo_{(N-k,k)}$ with $k-N \leq r \leq0$, $0 \leq s \leq k$. From the above discussion, it suffices to compare $(\Ff_{s} \ast \Ee_{0})\bo_{(k,N-k)}$ with $(\Ee_{0} \ast \Ff_{s})\bo_{(k,N-k)}$ for $1 \leq s \leq k$ and  $(\Ff_{0} \ast \Ee_{r})\bo_{(k,N-k)}$ with $(\Ee_{r} \ast \Ff_{0})\bo_{(k,N-k)}$ for $k-N \leq r \leq -1$.

First, we consider the case for $(\Ff_{s} \ast \Ee_{0})\bo_{(k,N-k)}$ and $(\Ee_{0} \ast \Ff_{s})\bo_{(k,N-k)}$ with $1 \leq s \leq k$. By definition and using the base change of the diagram \ref{diagram1}, we have 
\begin{align*}
	(\Ee_{0} \ast \Ff_{s})\bo_{(N-k,k)} &\cong  t_{*}(\pi_{13'}|_{\Zz'*})\ker(\rho_3)^{s} \cong t_{*}(\pi_{13'}|_{\Zz'*})g_{2*}g_{2}^{*}\ker(\rho_3)^s \cong t_{*}p_{*}\ker(\rho_3)^{s}, \\
	(\Ff_{s} \ast \Ee_{0})\bo_{(N-k,k)} &\cong
	t_{*}(\pi_{13}|_{\Zz*})\ker(\rho_2)^{s} \cong t_{*}(\pi_{13}|_{\Zz*})g_{1*}g_{1}^{*} \ker(\rho_2)^{s}\cong t_{*}p_{*}\ker(\rho_2)^{s}.
\end{align*}

For each $n \geq 1$, we have the following short exact sequence on $\Xx$
\begin{equation} \label{dses1} 
	0 \rightarrow \ker(\rho_2)^{n} \rightarrow \ker(\rho_3)^{n} \rightarrow \Oo_{nD} \otimes \ker(\rho_3)^{n} \rightarrow 0.
\end{equation}

Applying $t_{*}p_{*}$ to (\ref{dses1}) with $n=s$, we obtain the following exact triangle in $\Dd^b(\Gr(\GGG,k) \times \Gr(\GGG,k))$.
\begin{equation} \label{et1}
	t_{*}p_{*}\ker(\rho_2)^{s} \cong (\Ff_{s} \ast \Ee_{0})\bo_{(N-k,k)}  \rightarrow t_{*}p_{*}\ker(\rho_3)^{s} \cong 	(\Ee_{0} \ast \Ff_{s})\bo_{(N-k,k)}  \rightarrow 	t_{*}p_{*}(\Oo_{sD} \otimes \ker(\rho_3)^{s}).
\end{equation} 

We need to know the third term $t_{*}p_{*}(\Oo_{sD} \otimes \ker(\rho_3)^{s})$. When $s=1$, it is clear that $p_{*}(\Oo_{D} \otimes \ker(\rho_3)) \cong 0$ by the projective bundle formula. Here we use $\Oo_{D} \otimes \ker(\rho_3) = \Oo_{D} \otimes\Oo_{\PP_{\sub}(\Qq)}(-1)$ from the short exact sequence (\ref{dses'}). Thus we prove $(\Ff_{1} \ast \Ee_{0})\bo_{(N-k,k)} \cong  (\Ee_{0} \ast \Ff_{1})\bo_{(N-k,k)}$. Assuming $s \geq 2$ from now on.

Tensoring the short exact sequence (\ref{dses}) by $\ker(\rho_2)^{n-1}$ we get 
\[
0 \rightarrow \ker(\rho_2)^{n} \rightarrow \ker(\rho_3) \otimes \ker(\rho_2)^{n-1}  \rightarrow \Oo_{D} \otimes \ker(\rho_3) \otimes \ker(\rho_2)^{n-1} \rightarrow 0 .
\] Together with (\ref{dses1}), they form the following diagram of morphisms between exact triangles in $\Dd^b(\Xx)$.	
\[
\xymatrix{
	\ker(\rho_2)^{n} \ar[d]^{id} \ar[r] &\ker(\rho_3) \otimes  \ker(\rho_2)^{n-1}  \ar[d] \ar[r] &\Oo_{D} \otimes  \ker(\rho_3) \otimes  \ker(\rho_2)^{n-1} \ar[d] \\
	\ker(\rho_2)^{n}  \ar[d] \ar[r] &  \ker(\rho_3)^{n} \ar[d] \ar[r] &\Oo_{nD} \otimes  \ker(\rho_3)^{n} \ar[d] \\
	0 \ar[r] &\Oo_{(n-1)D} \otimes  \ker(\rho_3)^{n} \ar[r] &\Oo_{(n-1)D} \otimes   \ker(\rho_3)^{n}
}
\]

So we obtain the exact triangle in $\Dd^b(\Xx)$	
\begin{equation} \label{et2}
	\Oo_{D} \otimes  \ker(\rho_3) \otimes  \ker(\rho_2)^{n-1} \rightarrow \Oo_{nD} \otimes \ker(\rho_3)^{n} \rightarrow \Oo_{(n-1)D} \otimes  \ker(\rho_3)^{n}
\end{equation} for all $n \geq 1$ (Here we take $\Oo_{(n-1)D}$ to be $0$ when $n=1$).

Considering the exact triangle (\ref{et2}) with $2 \leq n=s \leq k$ and applying $p_{*}$ to it. Then by the projection formula, we have $p_{*}(\Oo_{D} \otimes  \ker(\rho_3) \otimes  \ker(\rho_2)^{s-1}) \cong 0$. This implies that $p_{*}(\Oo_{sD} \otimes \ker(\rho_3)^{s}) \cong  p_{*}(\Oo_{(s-1)D} \otimes  \ker(\rho_3)^{s})$.  Next, tensoring the exact triangle (\ref{et2}) with $n=s-1$ by the line bundle $\ker(\rho_3)$, then apply $p_{*}$. By the same argument as before, we obtain $p_{*}(\Oo_{(s-1)D} \otimes \ker(\rho_3)^{s}) \cong  p_{*}(\Oo_{(s-2)D} \otimes  \ker(\rho_3)^{s})$. Continuing this process, we will end up with 
\begin{equation*}
	p_{*}(\Oo_{sD} \otimes \ker(\rho_3)^{s}) \cong  p_{*}(\Oo_{(s-1)D} \otimes  \ker(\rho_3)^{s}) \cong \dots \cong p_{*}(\Oo_{2D} \otimes \ker(\rho_3)^{s}).
\end{equation*}

Finally, tenoring the exact triangle (\ref{et2}) with $n=2$ by the line bundle $\ker(\rho_3)^{s-2}$, we obtain 
\begin{equation*}
	\Oo_{D} \otimes  \ker(\rho_3)^{s-1} \otimes  \ker(\rho_2) \rightarrow \Oo_{2D} \otimes \ker(\rho_3)^{s} \rightarrow \Oo_{D} \otimes  \ker(\rho_3)^{s}.
\end{equation*} Since $s \geq 2$, applying $p_{*}$ and using the projection formula we get $p_{*}(\Oo_{D} \otimes  \ker(\rho_3)^{s-1} \otimes  \ker(\rho_2) ) \cong 0$ and thus
\begin{equation} \label{proj}
p_{*}(\Oo_{2D} \otimes \ker(\rho_3)^{s}) \cong p_{*}(\Oo_{D} \otimes  \ker(\rho_3)^{s}) \cong p_{*}(\Oo_{D} \otimes  \Oo_{\PP_{\sub}(\Qq)}(-s)).	
\end{equation}

Since $\rk \Qq=k$, using the projective bundle formula in (\ref{proj}), we conclude that : If $2 \leq s \leq k-1$, then $p_{*}(\Oo_{sD} \otimes \ker(\rho_3)^{s}) \cong 0$. So $(\Ff_{s} \ast \Ee_{0})\bo_{(N-k,k)} \cong  (\Ee_{0} \ast \Ff_{s})\bo_{(N-k,k)}$ for $1 \leq s \leq k-1$. If $s=k$, then $t_{*}p_{*}(\Oo_{D} \otimes  \Oo_{\PP_{\sub}(\Qq)}(-k)) \cong \Delta{*}\det(\Qq)[1-k]$ which is precisely the FM kernel $\Psi^{+}\bo_{(N-k,k)}$. Hence we obtain the exact triangle (\ref{eta}). 

Next, we prove the second case where we compare $(\Ff_{0} \ast \Ee_{r})\bo_{(k,N-k)}$ with $(\Ee_{r} \ast \Ff_{0})\bo_{(k,N-k)}$ for $k-N \leq r \leq -1$. We have to mention here that this case is not like the first case since the coherent sheaf $\Ss_{\pi}$ will be involved in the calculation, which needs some extra tool to help.

Similarly, by definition and using base change of diagram \ref{diagram1}, we have 
\begin{align*}
	(\Ee_{r} \ast \Ff_{0})\bo_{(N-k,k)} &\cong t_{*}p_{*}\ker(\rho_4)^{r}, \\
	(\Ff_{0} \ast \Ee_{r})\bo_{(N-k,k)} &\cong t_{*}p_{*}\ker(\rho_1)^{r}.
\end{align*}

Like (\ref{dses1}), for each $n \geq 1$ we have the following short exact sequence
\begin{equation} \label{sesd1}
	0 \rightarrow \ker(\rho_4)^{-n} \rightarrow \ker(\rho_1)^{-n} \rightarrow \Oo_{nD} \otimes \ker(\rho_1)^{-n} \rightarrow 0.
\end{equation} Applying $t_{*}p_{*}$ to it with $n=-r$ (note that $k-N \leq r \leq -1$), then we obtain the following exact triangle in $\Dd^b(\Gr(\GGG,k) \times \Gr(\GGG,k))$
\begin{equation} \label{et3}
	t_{*}p_{*}\ker(\rho_4)^{r} \cong (\Ee_{r} \ast \Ff_{0})\bo_{(N-k,k)} \rightarrow t_{*}p_{*}\ker(\rho_1)^{r} \cong (\Ff_{0} \ast \Ee_{r})\bo_{(N-k,k)} \rightarrow t_{*}p_{*}(\Oo_{(-r)D} \otimes \ker(\rho_1)^{r}).
\end{equation}

Thus it suffices to know the third term $t_{*}p_{*}(\Oo_{(-r)D} \otimes \ker(\rho_1)^{r})$. We also have the following exact triangles similar to (\ref{et2})  
\begin{equation} \label{et4}
		\Oo_{D} \otimes  \ker(\rho_1)^{-1} \otimes  \ker(\rho_4)^{-n+1} \rightarrow \Oo_{nD} \otimes \ker(\rho_1)^{-n} \rightarrow \Oo_{(n-1)D} \otimes  \ker(\rho_1)^{-n}
\end{equation} for all $n \geq 1$ (Here we take $\Oo_{(n-1)D}$ to be $0$ when $n=1$).

The argument is pretty much the same as the previous case except that now we have $\Oo_{D} \otimes  \ker(\rho_1)^{-1}= \Oo_{D} \otimes \Oo_{\PP_{\quo}(\Ss)}(-1)$ where $\Ss$ is only a coherent sheaf (not locally free in general), so we can not apply the usual projective bundle formula directly.

However, since $\Ss$ fits into the short exact sequence $0 \rightarrow \Ss \rightarrow \GGG \rightarrow \Qq \rightarrow 0$ on $\Xx$ and $\GGG$ has homological dimension $\leq 1$ by the assumption, plus the fact that $\Qq$ is locally free we conclude that $\Ss$ also has homological dimension $\leq 1$.

Thus we can apply the projective bundle formula for coherent sheaf with homological dimension $\leq 1$, i.e. Proposition \ref{Proposition}. Repeating the argument as above we also obtain the following isomorphism
\begin{equation} \label{proj'}
	p_{*}(\Oo_{(-r)D} \otimes \ker(\rho_1)^{r}) \cong p_{*}(\Oo_{2D} \otimes \ker(\rho_1)^{r}) \cong  p_{*}(\Oo_{D} \otimes  \ker(\rho_1)^{r}) \cong p_{*}(\Oo_{D} \otimes  \Oo_{\PP_{\quo}(\Ss)}(r)).	
\end{equation}

Observing that since we only consider $ -N+k \leq r \leq -1$ and $\rk \Ss=N-k$, the result of $p_{*}(\Oo_{(-r)D} \otimes \Oo_{\PP_{\quo}(\Ss)}(r))$ is exactly the same as the projective bundle formula when $\Ss$ is locally free of rank $N-k$. This implies that if $-k+N+1 \leq r \leq -1$, then $p_{*}(\Oo_{(-r)D} \otimes \Oo_{\PP_{\quo}(\Ss)}(r)) \cong 0$ and $(\Ee_{r} \ast \Ff_{0})\bo_{(N-k,k)} \cong (\Ff_{0} \ast \Ee_{r})\bo_{(N-k,k)}$. If $r=-k+N$, then  $p_{*}(\Oo_{(N-k)D} \otimes \Oo_{\PP_{\quo}(\Ss)}(-k+N)) \cong \det(\Ss)^{-1}[1+k-N]$, which is exactly $\Psi^{-}\bo_{(N-k,k)}$ and we get the second exact triangle (\ref{etb}). 
\end{proof}

Combining Lemma \ref{lemma 3}, Lemma \ref{lemma 4}, Theorem \ref{Theorem 8}, Theorem \ref{Proposition 3} and Theorem \ref{Theorem 6}, we prove Theorem \ref{Theorem 7}.

\begin{remark} \label{remark 2} 
Note that by Theorem \ref{Proposition 3} and Theorem \ref{Theorem 6}, the functors ${\F}_{\lambda}\bo_{(0,N)}$, ${\E}_{-\mu}\bo_{(N,0)}$ are not exceptional in general. More precisely, we have $\Hom({\E}_{-\mu}\bo_{(N,0)},{\E}_{-\mu}\bo_{(N,0)}) \cong \Hom(\bo_{(N,0)},\bo_{(N,0)})$, similarly for ${\F}_{\lambda}\bo_{(0,N)}$. In this geometric example where $\Kk(N-k,k)=\Dd^b(\Gr(\GGG,k))$, we have $\Kk(N,0)=\Dd^b(\Gr(\GGG,0))=\Dd^b(X)$ which is the base variety. Thus the functor $\bo_{(N,0)}$ is a FM transformation with FM kernel given by $\Delta_{*}\Oo_{X} \in \Dd^b(X \times X)$. We get 
\begin{align*}
	\Hom({\E}_{-\mu}\bo_{(N,0)},{\E}_{-\mu}\bo_{(N,0)}) & \cong \Hom(\bo_{(N,0)},\bo_{(N,0)})  \\
	& \cong \Hom_{\Dd^b(X \times X)}(\Delta_{*}\Oo_{X},\Delta_{*}\Oo_{X}) \\
	& \cong HH^{*}(X)
\end{align*}  which is the Hoschild cohomology of $X$.
\end{remark}

There is already much progress on constructing $\SOD$s for the derived category of coherent sheaves on relative (derived) Quot scheme of coherent sheaf with homological dimension $\leq 1$. 

For a precise formula, let $\sigma:\Ee^{-1} \rightarrow \Ee^{0}$ denote the morphism in the locally free resolution of $\GGG$ so that $\GGG_{\sigma} \coloneqq \GGG =\coker (\sigma)$. By taking dual we obtain a morphism $\sigma^{\vee}:\Ee^{0,\vee} \rightarrow \Ee^{-1,\vee}$, and the cokernel is given by the sheaf $\Hh_{\sigma^{\vee}} \coloneqq \coker(\sigma^{\vee})=\Ee xt_{\Oo_{X}}^1(\GGG,\Oo_{X})$. Note that $\Hh_{\sigma^{\vee}}$ is supported on the locus where $\GGG_{\sigma}$ is not locally free.  We also consider the relative Quot scheme $\Gr(\Hh_{\sigma^{\vee}},k)$. Then we have the following theorem

\begin{theorem}[Jiang-Leung \cite{JL}, Toda \cite{T}] \label{sod}
Assume the Tor-independent conditions holds for the pairs of integers $(k,k-i)$, where $0 \leq i \leq \mini\{k,N\}$ (see Definition 6.3 in \cite{Jia1}). Then there is a semiorthogonal decomposition of the form
\begin{equation} \label{eqsod}
\Dd^b(\Gr(\GGG_{\sigma},k)) = \langle {N \choose i}-copies \ of \  \Dd^b(\Gr(\Hh_{\sigma^{\vee}},k-i)): \  0\leq i \leq \mini\{k,N\}	 \rangle
\end{equation} which is called the \textit{projectivization formula} ($k=1$) in \cite{JL} and the \textit{Quot formula} (for general $k$) in \cite{T}.
\end{theorem} 

\begin{remark}
In fact, the above theorem can be formulated without the Tor-independent conditions in the derived algebraic geometry setting where $\Gr(\GGG_{\sigma},k)$ and $\Gr(\Hh_{\sigma^{\vee}},k-i)$ are quasi-smooth derived schemes over $X$, see \cite{T}.
\end{remark}

Although our approach to obtain a $\SOD$ is relatively elementary compared to those tools (e.g. Koszul duality, categorified Hall product) used in \cite{T}, we believe that the $\SOD$ in our result (Theorem \ref{Theorem 7}) is the same as the one in Theorem \ref{sod} except that the orthogonal complements $\Aa(N-k,k)$, $\Bb(N-k,k)$ are unclear in our result.

Note that in the case where $k<N$ and $i=k$, we have $\Gr(\Hh_{\sigma^{\vee}},0)=X$. So
there are $N \choose k$-copies of $\Dd^b(X)$ in the $\SOD$ (\ref{eqsod}).  On the other hand, from the $\SOD$s in Theorem \ref{Theorem 7}, the subcategories $\textnormal{Im}{\F}_{\blam}\bo_{(0,N)}$ and $\textnormal{Im}{\E}_{-\bmu}\bo_{(N,0)}$ are essential images of the highest and lowest weight categories $\Dd^b(\Gr(\GGG,N))$ and $\Dd^b(\Gr(\GGG,0))$, respectively. We also have $\Gr(\GGG,N)=\Gr(\GGG,0)=X$, so $\{\textnormal{Im}{\F}_{\blam}\bo_{(0,N)}\}_{\blam \in P(k,N-k)}$ and $\{\textnormal{Im}{\E}_{-\bmu}\bo_{(N,0)}\}_{\bmu \in P(N-k,k)}$ both give $N \choose k$-copies of of $\Dd^b(X)$. This part agrees with the $\Dd^b(X)$ part we mention first in the $\SOD$ (\ref{eqsod}).

Then the complement $\Aa(N-k,k)$, $\Bb(N-k,k)$ will be the rest that consists of  $N \choose i$-copies of  $\Dd^b(\Gr(\Hh_{\sigma^{\vee}},k-i))$ where $0\leq i \leq k-1$.
In future work, it would be interesting to see the  representation-theoretical meaning of the collection $\bigoplus_{k}\Aa(N-k,k)$, $\bigoplus_{k}\Bb(N-k,k)$ from the shifted $q=0$ affine algebra.


\begin{thebibliography}{99}
	
	\bibitem{AT}, N. Addington, R. Takahashi, \textit{A categorical $\SL_{2}$ action on some moduli space of sheaves}, arXiv:2009.08522v1 [math.AG]
	
	\bibitem{BLMHAP} A. Bayer, M. Lahoz, E. Macri,  H. Nuer, A. Perry and P. Stellari, \textit{Stability conditions in families}. arXiv:1902.08184
	
	\bibitem{Be} A. A. Beilinson, \textit{Coherent sheaves on $\PP^n$ and problems of linear algebra}, Functional Analysis and Its Applications, 12, (1978), no. 3, 68-69.
	
	\bibitem{BLM} A. A. Beilinson, G. Lusztig, R. MacPherson, \textit{A geometric setting for the quantum deformation of $\GLL_{n}$}, Duke Math. J. 61 (1990), 655-677.
	
	\bibitem{BLVdB} R.-O. Buchweitz, G.J. Leuschke, M. Van den Bergh, \textit{On the derived category of Grassmannians
	in arbitrary characteristic}, Compos. Math. 151 (2015), no. 7, 1242-1264.
	
	\bibitem{Bo}  A. Bondal, \text{Representations of associative algebras and coherent sheaves}, Izv. Akad. Nauk SSSR Ser. Mat. 53 (1989), no. 1, 25–44.
	
	\bibitem{BO} A. Bondal and D. Orlov. \textit{Semiorthogonal decomposition for algebraic varieties}, arxiv:9506012.
	
	\bibitem{BOR} P. Belmans, S. Okawa, and A. T. Ricolfi \textit{Moduli spaces of semiorthogonal decompositions in families},  arXiv:2002.03303 [math.AG]
	
	\bibitem{CK} S. Cautis, J. Kamnitzer, \textit{Braiding via geometric Lie algebra actions}, Compositio Math. 148 (2012), no. 2, 464–506.
	
	\bibitem{CKL1} S. Cautis, J. Kamnitzer, A. Licata, \textit{Coherent sheaves and categorical $\SL_2$ actions}, Duke Math. J. 154 (2010), no. 1, 135-179.
	
	\bibitem{CKL2} S. Cautis, J. Kamnitzer, A. Licata, \textit{Derived equivalences for cotangent bundles of Grassmannians via categorical $\SL_2$ actions}, J. Reine Angew. Math. 675 (2013), 53-99.
	
	\bibitem{CR} J. Chuang, R. Rouquier \textit{Derived equivalences for symmetric groups and $\SL_2$-categorification}, Ann. of Math. (2) 167 (2008), no. 1, 245-298.
	
	\bibitem{E} A. I. Efimov, \textit{Derived categories of Grassmannians over integers and modular representation theory}, Adv. in Math, Vol. 304, 2017, 179-226
	
	\bibitem{FH} W. Fulton and J. Harris. \textit{Representation Theory. A First Course}, Springer-Verlag, New York, (2004).
	
	\bibitem{FT} Michael Finkelberg, Alexander Tsymbaliuk, \textit{Multiplicative slices, relativistic Toda and shifted quantum affine algebras}, Representations and Nilpotent Orbits of Lie Algebraic Systems (special volume in honour of the 75th birthday of Tony Joseph), Progress in Mathematics 330 (2019), 133-304.
	
	\bibitem{Ha} R. Hartshrone, \textit{Algebraic geometry}, GTM volume 52, Springer-Verlag.
	
	\bibitem{Hsu} Y. H. Hsu. \textit{A categorical action of the shifted $q=0$ affine algebra}, 	arXiv:2009.03579 [math.RT]
	
	\bibitem{Huy} D. Huybrechts, \textit{Fourier-Mukai transforms in Algebraic Geometry}, Oxford University Press, (2006).
	
	\bibitem{Jia1} Q. Jiang, \textit{Derived categories of Quot schemes of locally free quotients, I}, arXiv:2107.09193v1 [math.AG] 
	
	\bibitem{Jia2} Q. Jiang, \textit{Derived projectivizations of complexes}, arXiv:2202.11636v2[math.AG] 
	
	\bibitem{JL} Q. Jiang, NC. Leung \textit{Derived category of projectivization and flops}, arXiv:2107.09193v1 [math.AG] 
	
	\bibitem{Kap85} M. M. Kapranov. \textit{On the derived categories of coherent sheaves on Grassman manifolds}, Izv. Akad. Nauk SSSR Ser. Mat. 48 (1984), 192-202; English translation in Math. USSR Izv. 24 (1985).
	
	\bibitem{Kap88} M. M. Kapranov. \textit{On the derived categories of coherent sheaves on some homogeneous spaces}, Invent. Math., 92, (1988), no. 3, 479-508.
	
	\bibitem{KL1} M. Khovanov, A. Lauda \textit{A diagrammatic approach to categorification of quantum groups I}, Represent. Theory 13 (2009), 309-347.
	
	\bibitem{KL2} M. Khovanov, A. Lauda \textit{A diagrammatic approach to categorification of quantum groups II}, Trans. Amer. Math. Soc. 363 (2011), 2685-2700.
	
	\bibitem{KL3} M. Khovanov, A. Lauda \textit{A diagrammatic approach to categorification of quantum groups III}, Quantum topology 1, Issue 1 (2010), 1-92.
	
	\bibitem{KP} A. Kuznetsov, A.Polishchuk, \textit{Exceptional collections on isotropic Grassmannians},  J. Eur. Math. Soc. 18.3 (2016), 507–574.

	\bibitem{Ku} A. Kuznetsov, \textit{Base change for semiorthogonal decompositions}, Compos. Math. 147 (2011), no. 3, 852–876.
	
	\bibitem{Ku1} A. Kuznetsov, \textit{Semiorthogonal decompositions in algebraic geometry}, Proceedings of the International Congress of Mathematicians—Seoul 2014. Vol. II, Kyung Moon Sa, Seoul, 2014, pp. 635–660. 
	
	\bibitem{Ku2} A. Kuznetsov, \textit{Semiorthogonal decompositions in families}, Proceedings of the International Congress of Mathematicians— 2022.

	\bibitem{Lu} G. Lusztig, \textit{Introduction to quantum groups}. Birkhauser, Boston, 1993.

	\bibitem{N} A. Negut, \textit{Hecke correspondence for smooth moduli spaces of sheaves}, arXiv:1804.03645 

	\bibitem{O} D. Orlov, \textit{Projective bundles, monoidal transformations, and derived categories of coherent sheaves}, Izv. Ross. Akad. Nauk Ser.Mat. 56 (1992), no. 4, 852–862.
	
	\bibitem{R} R. Rouquier, \textit{2-Kac-Moody algebras}, arXiv:0812.5023 [math.RT]
	
	\bibitem{T} Y. Toda, \textit{Derived categories of Quot schemes of locally free quotients via categorified Hall prodcuts}, arXiv:2110.02469v2 [math.AG]
	
	\bibitem{Z} Y. Zhao, \textit{A categorical quantum toroidal action on the Hilbert schemes} arXiv:2009.11267v2 [math.AG]
\end{thebibliography}
\end{document}